%% file: tambara.tex
\documentclass[reqno,dvipsnames]{amsart}
\usepackage{fullpage}
\usepackage{todonotes}
\usepackage[all,2cell]{xy}
\usepackage{graphicx, amsmath, amssymb, amsthm}
\usepackage{newclude}
\usepackage{mathtools} % also loads amsmath
\usepackage{mathrsfs}
\usepackage[hidelinks,backref]{hyperref}
\usepackage{enumitem}
\usepackage[normalem]{ulem}
\usepackage{halloweenmath}
\usepackage{listings}
\lstdefinelanguage{Sage}[]{Python}
{morekeywords={False,sage,True},sensitive=true}
\definecolor{dblackcolor}{rgb}{0.0,0.0,0.0}
\definecolor{dbluecolor}{rgb}{0.01,0.02,0.7}
\definecolor{dgreencolor}{rgb}{0.2,0.4,0.0}
\definecolor{dgraycolor}{rgb}{0.30,0.3,0.30}

\DeclareMathOperator{\dlog}{dlog}
\newcommand{\Cpoo}{C_{p^\infty}}
\newcommand{\subconj}{\precsim}
\newcommand{\trans}{\pitchfork}

\newcommand{\Tamb}{\mathrm{Tamb}}

\input{macros.tex}

\makeatletter
\providecommand\@dotsep{5}
\renewcommand{\listoftodos}[1][\@todonotes@todolistname]{%
  \@starttoc{tdo}{#1}}
\makeatother

\SelectTips{cm}{}
\UseAllTwocells
\NoCompileMatrices

\begin{document}

\title{
  Prisms and Tambara functors I:\\
  Twisted powers, transversality, and the perfect sandwich
}

\author[Y.~J.~F.~Sulyma]{Yuri~J.~F. Sulyma}
\email{yuri.sulyma@protonmail.com}

\begin{abstract}
  We construct a faithful and conservative functor from prisms to $C_{p^\infty}$-Tambara functors; in appropriate situations, this gives an algebraic description of $\underline\pi_0{\mathrm{TC}^-}$. We also present two integral variants using the generalized $n$-series of Devalapurkar-Misterka. The construction is based on the ``twisted $I$-adic'' or ``$(p^\bullet)_q$'' filtration, and is closely related to $q$-divided powers. To verify the axioms, we introduce a new technique for constructing Tambara functors, inspired by transversal prisms. We apply this to give a conceptual construction of Molokov's de Rham-Witt comparison map, and generalize it to a triangle sandwiching prismatic theory between theories built from Witt vectors and adjunction of $p$-power roots.
\end{abstract}

\maketitle
\tableofcontents

\include*{intro}
\include{background}
\include{main}
\include{drw}
\include{gns}

% \bibliographystyle{amsalpha}
% \bibliography{../bibliography}
\input{tambara.bbl}

\end{document}

%% file: macros.tex
\usepackage{relsize}
\usepackage[bbgreekl]{mathbbol}
\usepackage{amsfonts}
\DeclareSymbolFontAlphabet{\mathbb}{AMSb} %to ensure that the meaning of \mathbb does not change
\DeclareSymbolFontAlphabet{\mathbbl}{bbold}
\newcommand{\prism}{{\mathlarger{\mathbbl{\Delta}}}}

\let\>=\undefined
\newcommand{\>}{\right\rangle}
\newcommand{\psr}[1]{[\![#1]\!]}

\newcommand{\lf}{\left\lfloor}
\newcommand{\rf}{\right\rfloor}

\newcommand{\RU}{\mathit{RU}}
\newcommand{\m}{\underline}
\newcommand{\ol}{\overline}

\newcommand{\mpi}{\m\pi}

\newcommand{\blue}[1]{\textcolor{blue}{#1}}
\newcommand{\red}[1]{\textcolor{red}{#1}}
\newcommand{\green}[1]{\textcolor{Green}{#1}}
\newcommand{\gray}[1]{\textcolor{black!30}{#1}}

\newcommand{\id}{\mathrm{id}}

\renewcommand{\~}{\widetilde}

\newcommand{\defeq}{\mathrel{:=}}

\newcommand{\tops}{\texorpdfstring}

\DeclareMathOperator{\res}{res}
\DeclareMathOperator{\tr}{tr}
\newcommand{\can}{\mathrm{can}}

\newdir{ >}{{}*!/-10pt/@{>}}
\newcommand{\pullback}{\ar@{}[dr]|<<{\lrcorner}}
\newcommand{\pushout}{\ar@<2pt>@{}[ul]|<{\mbox{\Huge$\ulcorner$}}}

\newcommand{\adjnctn}[4]{\xymatrix@1{
  #1 \ar@<1ex>[r]^-{#3} \ar@{}[r]|-{\bot} & #2 \ar@<1ex>[l]^-{#4}
}}

\newcommand{\mz}[1]{(\!(#1)\!)}
\newcommand{\mup}[1]{\ar@/_1em/[u]_-{#1}}
\newcommand{\mdown}[1]{\ar@/_1em/[d]_-{#1}}
\newcommand{\mt}{\ar@{-}[d]}

\newcommand{\CRing}{\mathrm{CRing}}

\newcommand{\E}{\mathbf E}
\newcommand{\F}{\mathbf F}
\newcommand{\cF}{\mathcal F}

\newcommand{\G}{\mathbf G}
\renewcommand{\H}{\mathrm H}
\newcommand{\bfH}{\mathbf H}

\newcommand{\N}{\mathbf N}
\renewcommand{\O}{\mathcal O}
\renewcommand{\P}{\mathcal P}
\newcommand{\bfP}{\mathbf P}
\newcommand{\Q}{\mathbf Q}

\let\sec=\S
\renewcommand{\S}{\mathbf S}
\newcommand{\T}{\mathbf T}
\newcommand{\W}{\mathbf W}
\newcommand{\Z}{\mathbf Z}

\newcommand{\BK}{\mathfrak S}

\newcommand{\WCart}{\mathrm{WCart}}
\newcommand{\g}{d}
\newcommand{\Ainf}{A_{\mathrm{inf}}}

\newcommand{\sfS}{\mathsf S}
\newcommand{\sfT}{\mathsf T}

\renewcommand{\hat}{\widehat}

\newcommand{\Fil}{\mathrm{Fil}}
\newcommand{\tnsr}{\otimes}
\newcommand{\ctnsr}{\mathbin{\widehat\otimes}}

\DeclareMathOperator{\Hom}{Hom}
\DeclareMathOperator{\End}{End}
\DeclareMathOperator{\Aut}{Aut}
\DeclareMathOperator{\Tor}{Tor}

\DeclareMathOperator{\Spf}{Spf}
\DeclareMathOperator{\THH}{THH}
\DeclareMathOperator{\TC}{TC}
\newcommand{\TCmin}{\TC^-}
\DeclareMathOperator{\TF}{TF}
\DeclareMathOperator{\TP}{TP}
\DeclareMathOperator{\TR}{TR}

\newcommand{\cont}{\mathrm{cont}}
\DeclareMathOperator{\im}{im}

\newcommand{\morph}{\mathop{\longrightarrow}\limits}
\newcommand{\xra}{\xrightarrow}

\renewcommand{\lim}{\mathop{\operatorname*{lim}\limits_{\longleftarrow}}\limits}

\newcommand{\lcm}{\mathop{\operatorname{lcm}}\limits}

\newcommand{\mono}{\hookrightarrow}

\newcommand{\iso}{\overset\sim\longrightarrow}

\newcommand{\isom}{\cong}

\newtheorem{theorem}{Theorem}[section]
\newtheorem*{theorem*}{Theorem}

\newtheorem{lemma}[theorem]{Lemma}
\newtheorem{corollary}[theorem]{Corollary}
\newtheorem{proposition}[theorem]{Proposition}

\theoremstyle{definition}
\newtheorem{definition}[theorem]{Definition}
\newtheorem{conjecture}[theorem]{Conjecture}
\newtheorem{remark}[theorem]{Remark}
\newtheorem{example}[theorem]{Example}
\newtheorem{warning}[theorem]{Warning}

\newtheorem{question}[theorem]{Question}
\newtheorem{construction}[theorem]{Construction}

%% file: intro.tex
% !TeX root = tambara.tex
\section{Introduction}
\label{sec:intro}

This paper concerns the interaction between arithmetic geometry and equivariant algebra. We will explain what we are doing first from an arithmetic perspective, then from the equivariant perspective. The necessary background can be found in \cite[Part 1]{SulSliceTHH} as well as \sec\ref{sec:background}.

% Arithmetic perspective
\subsection{Arithmetic explanation}

A prism $(A, I)$ comes with many interesting filtrations. In the mainstream approach to prismatic cohomology, the emphasis is on the $I$-adic filtration and its Frobenius pullback, the Nygaard filtration. In this paper, we study the ``twisted $I$-adic filtration'', which is the filtration by the ideals
\[ I_n \defeq I\phi(I)\dotsm\phi^n(I). \]
Over the $q$-de Rham prism, this amounts to looking at $(p^i)_q$ rather than $(p)_q^i$. This filtration is used, for example, in defining Breuil-Kisin twists and the prismatic logarithm \cite[\sec2]{APC}.

We assert that this filtration is best understood through the lens of equivariant algebra (meaning $\pi_0$ of equivariant homotopy theory). Specifically, we show that the quotients $A/I_n$ fit together into a \emph{Tambara functor} (for the group $G=\Cpoo$). We review Tambara functors in \sec\ref{sec:background}; for now, let us just say that
\begin{itemize}
   \item a $G$-Tambara functor $\m A$ consists of commutative rings indexed by the subgroups $H\le G$, written $\m A(G/H)$ or $A^H$, along with various maps between them;
   \item Tambara functors are the equivariant analogue of commutative rings;
   \item free Tambara functors (in a suitable sense) are given by Witt vectors. 
\end{itemize}

The Tambara functor structure on Witt vectors amounts to the $F$, $V$, $N$ maps
\[\xymatrix{
   W_{n-1}(R) \ar@<1ex>[r]^-V \ar@<-1ex>[r]_-N & W_n(R) \ar[l]|-F
}\]
given in ghost components\footnote{We index our Witt vectors such that $W_n=\W_{p^n}$, i.e.\ starting from $0$, e.g.\ $W_1(\F_p)=\Z/p^2$.} by
\begin{align*}
   F(w_0,\dotsc,w_n) &= (w_1,\dotsc,w_n)\\
   V(w_0,\dotsc,w_{n-1}) &= (0, pw_0, \dotsc, pw_{n-1})\\
   N(w_0,\dotsc,w_{n-1}) &= (w_0, w_0^p, \dotsc, w_{n-1}^p)
\end{align*}
The $N$ map is less well-known than the $F$ and $V$ maps; it was originally described by Angeltveit \cite{AngeltveitNorm}. Let us note that $N$ satisfies
\begin{align*}
   FN(x) &= x^p\\
   N(x) &= x \bmod V,
\end{align*}
which characterize the Norm in ghost components, and that $N^n$ is the familiar multiplicative lift from $W_0$ to $W_n$. Like $F$ and $V$, $N$ exists as a map $N\colon W(R)\to W(R)$, where it is given by
\[ N(x) = x - V\delta(x). \]
This formula is due to Borger \cite{AngeltveitNorm}.

When $(A, I)$ is a perfect prism, corresponding to a perfectoid ring $R=A/I$, the Tambara functor $\m A$ we associate\footnote{Warning: in the literature, $\m A$ is usually used to denote the Burnside Mackey functor. We denote this by $\m W(\Z)$.} to $A$ is the Witt vectors Tambara functor $\m W(R)$. To extend this to general prisms, we translate the above formulas for $F$, $V$, and $N$ from $W(R)$ to $\Ainf(R)$, and show that the resulting formulas work for all prisms\footnote{We sketched this construction at the end of \cite[\sec3.3]{SulSliceTHH}, but we did not have the technology then to verify the Tambara axioms, nor to extend the construction to non-orientable prisms.}. The ``$F$'' and ``$V$'' parts of this story are more-or-less well-known to experts; the point is really the $N$ maps, which are more difficult to construct.

To understand the meaning of the $N$ maps, we turn to a variant of our construction. Consider the ring $D=\hat\Z\psr{q-1}$; this is a $\lambda$-ring with Frobenius lifts $\psi^n(f) = f(q^n)$, and contains the elements $\mu=q-1$ as well as the $q$-analogues\footnote{The notation $[n]_q$ is more common. We follow \cite{GLQqCrys} in using $(n)_q$ so that $f^{[n]_q}$ can be used for a $q$-divided power.}
\[ (n)_q \defeq \frac{q^n-1}{q-1}=\frac{\psi^n(\mu)}\mu. \]
We will study the structure of the quotients:
\begin{align*}
   \m D^-(G/C_n) &= D/\psi^n(\mu) = D/(q^n-1)\\
   \m D(G/C_n) &= D/(n)_q
\end{align*}
We show that $\m D^-$ forms a Tambara functor for $G=\Q/\Z$, and that $\m D$ is almost a Tambara functor. From the $q$-perspective, what we are doing is making sense of the symbol (called a \emph{$q$-twisted power} in \cite{GLQqCrys})
\[ f^{(m)_{q^n}} \]
for $m,n\in\N$. Essentially, we show it can be interpreted as a special case of the symbol $f^{C_{mn}/C_n}$, which already has a meaning in equivariant algebra\footnote{To avoid confusion, we note that this is not literally true with our choice of notation: for us, $f^{C_{mn}/C_n}$ is an operation on rings with $C_{mn}$-action, while $N^{mn}_n$ is a genuine analogue of this operation. We are saying that $f^{(m)_{q^n}}$ is a special case of $N^{mn}_n(f)$; since our actions are trivial, $f^{C_{mn}/C_n}$ in this case would be just $f^m$.}.

We note there are some existing candidates for this expression: 
\begin{itemize}
   \item for any $f$, the expression
   \[ \~N_1^p(f)\defeq\phi(f)-(p)_q\delta(f) \]
   is a good candidate for $f^{(p)_q}$. More generally,
   \[ \~N_{p^{n-1}}^{p^n}(f)\defeq\phi(f)-(p)_{q^{p^n}}\delta(f) \]
   is a good candidate for $f^{(p)_{q^{p^n}}}$.

   \item if $x$ and $y$ are of rank one (meaning $\delta(x)=\delta(y)=0$), then
   \[ (x-y)^{(m)_q} \defeq \prod_{i=0}^{m-1} (x-q^i y) \]
   is a good candidate for the left-hand side. More generally,
   \[ (x-y)^{(m)_{q^n}} \defeq \prod_{i=0}^{m-1} (x-q^{ni} y) \]
   is a good candidate for the left-hand side.
\end{itemize}
These options are not entirely satisfactory:
\begin{itemize}
   \item $(x-y)^{(m)_{q^n}}$ depends on both $x$ and $y$, not just on $x-y$;
   \item for $x$ and $y$ of rank one, $\~N_1^p(x-y)$ and $(x-y)^{(p)_q}$ are generally different;
   \item the second formula only works for differences of elements of rank one;
   \item it is not known how to extend the first formula to $f^{(m)_{q^n}}$ for general values of $m$ and $n$.
\end{itemize}
We offer an operation $N^{mn}_n$ with the following properties:
\begin{itemize}
   \item it is a multiplicative map
   \[ N^{mn}_n\colon D/(q^n-1) \to D/(q^{mn}-1); \]

   \item $N^{mn}_n(f)$ ``interpolates'' from $f^m$ to $\psi^m(f)$ via the various $\psi^g(f)^{m/g}$, and is defined by this property;

   \item for $x$ and $y$ of rank one, $(x-y)^{(m)_{q^n}}$ is a \emph{lift} of $N^{mn}_n(x-y)$ (Theorem \ref{thm:s-lift}). In particular, $(x-y)^{(m)_{q^n}}\bmod (q^{mn}-1)$ depends only on $(x-y)\bmod (q^n-1)$;

   \item for any $f$, $\~N_1^p(f)$ is a \emph{lift} of $N^p_1(f)$. We also give conjectural formulas for an operation $\~N^{mn}_n$ lifting $N^{mn}_n$ for all $m,n$, generalizing the expression for $\~N^p_1$ above (Theorem \ref{thm:vartheta-lift}). We prove the conjecture when $m$ is a power of a prime or the product of two primes (Proposition \ref{prop:vartheta-pq}).
\end{itemize}

This construction can be applied more generally to $q$-pd thickenings \cite[Definition 16.2]{Prismatic} (which we define integral versions of). For a $q$-pd thickening $(D,I)$, our norm $N^p_1$ is given by
\begin{align*}
   N^p_1\colon D/I &\to D/(p)_q I\\
   N^p_1(f) &= \phi(f) - (p)_q\delta(f)
\end{align*}
The well-definedness of this map is part of the \emph{definition} of $q$-pd thickening \cite[Definition 3.1]{GLQqCrys} (see also the Remarks following that Definition for a comparison with \cite[Definition 16.2]{Prismatic}). One can also replace $q$-analogues with any generalized $n$-series (GNS) in the sense of Devalapurkar-Misterka \cite{GNS} satisfying the assumptions of the $s$-Lucas theorem \cite[Theorem 2.4.2]{GNS}.

\begin{question}
   What can the lifts $\~N^{mn}_n$ or $(x-y)^{(m)_{q^n}}$ do that the norm $N^{mn}_n$ cannot? Are they simply extraneous data? Or do they play a similar role to Frobenius lifts?
\end{question}

The object $\m D$ is obtained from $\m D^-$ by killing the value of $\m D^-$ at the trivial subgroup, $D/\mu$, \emph{as a Mackey functor} (= equivariant abelian group). This process does not preserve norm maps: that is, the Norm generally does not descend to a map as follows:
   \[\xymatrix{
      D/(q^n-1) \ar[d] \ar[r]^-{N^{mn}_n} & D/(q^{mn}-1) \ar[d] \\
      D/(n)_q \ar@{..>}[r]_-{N^{mn}_n}|-\times & D/(mn)_q
   }\]
   Concretely, $N^6_2$ does not descend to a map $D/(2)_q\to D/(6)_q$. Therefore, the quotients $D/(n)_q$ do not form a Tambara functor.

   Blumberg and Hill introduced \emph{incomplete} Tambara functors to handle situations like this \cite{IncompleteTambara}. The only Norms that survive the quotienting procedure are $N^{p^{m+n}}_{p^n}$. However, $N^6_2$ \emph{does} descend to a map
   \[ N^6_2\colon D/(2)_q \to D/(6)_q[(3)_q^{-1}], \]
   and the existing framework of incomplete Tambara functors does not know about this map. In future work, we intend to give a more permissive framework that captures the structure formed by the $D/(n)_q$.

\begin{remark}
   Here is a rough legend of how arithmetic information is organized in the Tambara formalism.
   \begin{itemize}
      \item for the $\m A$ construction, the Mackey functor axiom $FV=p$ is equivalent to the ``prism condition'' $p\in I+\phi(I)A$. For the $\m D^-$ construction, this axiom amounts to $(n)_{q^m} \equiv n \bmod \psi^m(\mu)$. Note that $(D,\mu)$ is not a prism.

      \item since $\m\W(\Z)$ is the initial Tambara functor, $A/I\phi(I)$ and $D^-/(q^p-1)$ are $W_1(\Z)=\Z[x]/(x^2-px)$-algebras. This captures the congruences $\phi(I)^2 \equiv p\phi(I) \bmod I_1$ and $(p)_q^2 \equiv p(p)_q \bmod (q^p-1)$.

      \item consider the operation $\~N^p_1(f) = \phi(f) - (p)_q\delta(f)$ for $A=(\Z_p[q^{1/p^\infty}]^\wedge_{(p,q-1)}, (p)_{q^{1/p}})$, or any prism under it. The axioms of Tambara functors require that
      \[
         \~N^p_1(f) \equiv
         \begin{cases}
            f^p & \bmod (p)_{q^{1/p}}\\
            \phi(f) & \bmod (p)_q
         \end{cases}
      \]
      However, $\m A$ satisfies a stronger version of this congruence than a generic Tambara functor would:
      \[
         \~N^p_1(f) \equiv
         \begin{cases}
            f^p & \bmod (p)_{q^{1/p}}^p\\
            \phi(f) & \bmod (p)_q
         \end{cases}
      \]
      We refer to this phenomenon as \emph{refraction}, and say that $\m A$ is a \emph{refractive} Tambara functor. The exact definition of refractive is still in progress, but note that $\m D^-$ is not refractive, since $\deg(q-1)^p>\deg(p)_q$.

      In particular, we see that $\~N^p_1(f)$ is divisible by $(p^2)_{q^{1/p}}!$ whenever $f$ is divisible by $(p)_{q^{1/p}}$. This is essentially the ``fundamental lemma of $q$-crystalline cohomology'' \cite[Lemma 16.7]{Prismatic}.

      \item we advertise that equivariant homotopy theory can also produce the filtration
      \[ \Fil^i D = (pi)_q! D; \]
      see Figure \ref{fig:filtrations}. This is the main result of \cite{SulSliceTHH}, which studies an equivariant filtration called the \emph{slice filtration} on $\THH$ of perfectoid rings. (Since we have not yet globalized the results of that paper, this part is a bit impressionistic.) The slice filtration is the minimal filtration such that the norms $N^{mn}_n(f)$, i.e.\ $f\mapsto f^{(m)_{q^n}}$, scale slice filtration by $m$, just as ordinary powers $f^m$ scale Postnikov filtration by $m$. This additional scaling property of $N^{mn}_n$ seems to correspond to a lemma used for the convergence of the $q$-logarithm/prismatic logarithm \cite[Proposition 4.9]{AClBTrace}.
   \end{itemize}
\end{remark}

\begin{figure}
   \label{fig:filtrations}
   \begin{tabular}{lcc} 
      Name & $\Fil^i \Z_p\psr{q-1}$ & $\Fil^i A$\\\hline
   $I$-adic & $(p)_q^i$ & $I^i$\\
      twisted $I$-adic & $(p^i)_q$ & $I\phi(I)\dotsm\phi^{i-1}(I)$\\
      factorial & $(pi)_q!$ & $\prod_{k=0}^\infty \phi^k(I)^{\lf i/p^k\rf}$
   \end{tabular}
   \caption{Filtrations on the $q$-de Rham and general prisms}
\end{figure}

While this is all very conceptually appealing, it does not yet connect very concretely to existing work in prismatic cohomology. Our main application, which should give a sense of what this perspective is good for, is to de Rham-Witt theory. Molokov has constructed a natural comparison map
\[ W_m(A/I) \to A/I_m \]
for any prism $(A,I)$. Since Witt vectors are free Tambara functors, our results immediately give a new construction of this map, and characterize it by a universal property. In fact, we generalize this map in two ways: we show that there are maps
\begin{equation}
   \label{eq:sandwich}
   W_m(A/I_n) \to A/I_{m+n} \to W_m(A/\phi^m(I_n))
\end{equation}
for all $m,n\ge0$. We refer to this diagram as the \emph{perfect sandwich} since the outer terms are related to inverse and direct perfection, respectively. The second map here is induced by the map $A\to W(A)$ encoding the $\delta$-structure on $A$, so this triangle plays two different universal properties of Witt vectors off of each other.

The author has repeatedly claimed that the composite in \eqref{eq:sandwich} is $W_m(\phi^m)$. At the eleventh hour, we realized that this is false for $n>0$. Write the composite as $W_m(\phi^m) + \varepsilon_m$ for some correction term $\varepsilon_m$. We show that $\varepsilon_m$---which by definition is a difference of ring homomorphisms---vanishes when $n=0$, in its last coordinate, and on the image of $A\to W(A)\to W_m(A/I_n)$ (Theorem \ref{thm:epsilon-vanish}). This last property tells us that $\varepsilon_m$ measures the difference between $\delta_{W(-)}$ and $\delta_A$. We believe this $\varepsilon_m$ deserves further study.

\begin{remark}
   The second map also gives a concrete description of the quotients $A/I_n$: for example, over the Breuil-Kisin prism we have
\[ \BK/I_2 = \{(w_0,w_1,w_2)\in W_2(\O_K[\pi^{1/p^2}]) \mid w_2\in \O_K,\, w_1\in\O_K[\pi^{1/p}]\}. \]
We interpret this as saying that, compared to the perfectoid theory which adds infinitely much ramification, the prismatic theory adds ``just the right amount'' of ramification.
\end{remark}

% Equivariant explanation
\subsection{Equivariant explanation}

We now turn to the equivariant perspective. In many cases (but not all \cite{LurieMOPrisms}), relative prismatic cohomology over $A$ is related to relative topological Hochschild homology $\THH(-/\S_A)$ relative to a spherical lift $\S_A$ of $A$. When $A$ is a perfect prism, we can take $\S_A = W^+(A/p)$, where $W^+$ is the spherical Witt vectors (\cite[Example 5.2.7]{Ell2}, \cite[\sec2.1]{ChromaticNullstellensatz}, \cite{AntieauSWitt}); for the Breuil-Kisin prism $\BK=W(k)\psr z$, we can take $\S_\BK=W^+(k)\psr z$. Topological negative cyclic homology $\TCmin(-/\S_A)$ arises as the $\T$-fixed points of a genuine equivariant commutative ring spectrum
\[ T^h(-/\S_A)=T(-/\S_A)^{E\T_+}. \]
What we are doing is providing an algebraic description of $\mpi_0 T^h(\bar A/\S_A)$; we will abusively denote this by $\mpi_0{\TCmin(\bar A/\S_A)}$. However, we do not verify that our construction has the desired properties\footnote{It is known to work for perfect prisms, which follows by Angeltveit's work \cite{AngeltveitNorm}. It may be possible to reduce to this case, but functoriality of spherical lifts necessary to make this argument work is not yet known.}. We can do something similar for $\mpi_0{\TP}$, but the result is now an \emph{incomplete} Tambara functor \cite{IncompleteTambara}.

Since we generally do not have a cyclotomic lift available, we have to construct our Tambara functor by hand. This poses a technical challenge: while \emph{Mackey} functors can be constructed by hand, this is generally not feasible for Tambara functors. The reason is that the ``Tambara reciprocity'' identities which tell us how to move a norm past a sum or transfer are difficult even to generate, let alone verify. To get around this, we introduce \emph{transversal} Mackey and Tambara functors, inspired by transversal prisms.

The idea is as follows. Viewing a Mackey functor $\m M$ as a $G$-spectrum, we can reconstruct $\m M$ from its geometric fixed-point spectra $\{\Phi^H\m M\}_{H\le G}$ by work of Glasman and Ayala--Mazel-Gee--Rozenblyum \cite{Glasman,AMGRStratified}. \emph{Transversal} Mackey functors are essentially those $\m M$ which can be reconstructed from $\{\mpi_0\Phi^H\m M\}_{H\le G}$; in other words, they support a version of the stratified story that stays in the world of Mackey functors. This makes it feasible to construct norm maps (and verify Tambara reciprocity) by hand. 

\begin{example}
   A $C_p$-Tambara functor is defined to be transversal when the diagram
   \[\xymatrix{
      A^{C_p} \ar[r] \ar[d] & A^{\Phi C_p} \ar[d]\\
      \H^0(A^e) \ar[r] & \hat\H^0(A^e)
   }\]
   is a pullback, where $A^{\Phi C_p}\defeq A^{C_p}/\tr(A^e)$. A transversal $C_p$-Tambara functor is thus equivalent to the following data:
   \begin{itemize}
      \item a ring with $C_p$-action $A=A^e\in\CRing^{BC_p}$;
      \item a plain ring $B=A^{\Phi C_p}\in\CRing$;
      \item a factorization (via ring maps)
      \[\xymatrix{
         & B \ar@{..>}[d]\\
         A \ar@{..>}[ur]^-\phi \ar[r]_-{x^{C_p/e}} & \hat\H^0(C_p, A)
      }\]
      of $A$'s ``Tate-valued Frobenius''\footnote{This version of the Tate-valued Frobenius uses the $C_p$-action on $A$; see \cite[Remark IV.1.13]{NikolausScholze}.} through $B$. The map $\phi$ must be $C_p$-equivariant with respect to the trivial action on $B$ (this is automatic if $B\to\hat\H^0(A)$ is injective, but that is usually not the case).
      \end{itemize}
      Given this data, we define $A^{C_p}$ by pullback as above, and then $V$ and $N$ are induced by
      \[
         \vcenter{\xymatrix{
            A \ar[r]^-0 \ar[d]_-{C_p/e\cdot x} & B \ar[d]\\
            \H^0(A) \ar[r] & \hat\H^0(A)
         }}
         \quad\text{and}\quad
         \vcenter{\xymatrix{
            A \ar[r]^-\phi \ar[d]_-{x^{C_p/e}} & B \ar[d]\\
            \H^0(A) \ar[r] & \hat\H^0(A)
         }}
      \]
      respectively. This should be compared with the recent work of Yang \cite{NormedEooRingsCp}, which proves a similar characterization of $C_p$-$\E_\infty$-ring \emph{spectra}.

      The example to keep in mind is $A^{C_p}=W_1(A)$ is the Witt vectors of a $p$-torsionfree ring $A$ with trivial $C_p$-action, in which case $A=B$ and $\phi$ is the identity.
\end{example}

In this paper, we employ this technique in an ad hoc way. We hope to develop the theory of transversal Tambara functors systematically in future work.

For $n=0$, the first map in the perfect sandwich \eqref{eq:sandwich} is the map
\begin{align*}
   W_n(A/I) &= \pi_0\THH(A/I)^{C_{p^n}}\\
   &\isom \pi_0\THH(A/I|\S_A)^{C_{p^n}}\\
   &\to \pi_0\THH(A/I|\S_A)^{hC_{p^n}}\\
   &= A/I_n
\end{align*}
from genuine to homotopy fixed points. Taking the limit over $F$, we obtain the map
\[ W((A/I)^\flat) = \pi_0\TF(A/I|\S_A) \to \pi_0\TCmin(A/I|\S_A) = A \]
from $\TF$ to $\TCmin$. For example, over the Breuil-Kisin prism this is the inclusion
\[ W(\O_K^\flat) = W(k)\mono W(k)\psr z = \BK. \]
The second map in the perfect sandwich is somehow related to $\TR$, but we do not yet know the precise statement here. Thus, the perfect sandwich is relating two flavors of equivariant homotopy: one indexed on subgroups of compact Lie groups ($\TF$), the other indexed on quotients of profinite groups ($\TR$).

Of course, the most exciting aspect of our work is the possibility of contemplating $G$-prisms and $G$-prismatic cohomology for more general $G$. In order to do this, we need to identify the image of the functor $A\mapsto\m A$, as well as the additional structure on $\m A$ needed to reconstruct the $\delta$-structure on $A$. These turn out to be very difficult, and will be the subject of a sequel joint with Alex Frederick. We do not take a position on how generally $G$-prismatic cohomology can be defined, but are confident that $G$-crystalline cohomology can be defined for any $G$ (and should recover $q$-crystalline cohomology when $G=\Q/\Z$).

\subsection{Notation}
\label{sub:notation}

If $A$ is a ring with $G$-action, and $g\in G$, $x\in A$, we write either $g\cdot x$ or $x^g$ to denote the action of $g$ on $x$. If $K\le H$ and $x\in A^K$, we define
\begin{align*}
   H/K\cdot x &= \sum_{hK\in H/K} h\cdot x\\
   x^{H/K} &= \prod_{hK\in H/K} x^h
\end{align*}

We write $\phi$ for the Frobenius on a prism, and $\psi^m$ for the Frobenii on a $\lambda$-ring. Similar to the $x^h$ notation, we sometimes write $f^\phi$ or $f^{\psi^m}$ instead of $\phi(f)$ or $\psi^m(f)$.

Since the author is by now fully $W$-pilled, we will write $F$ and $V$ for the restriction and transfer in an arbitrary Tambara functor. In the $p$-typical case, $S^n_k$ means $S^{C_{p^n}}_{C_{p^k}}$, for $S\in\{F,V,N\}$. In the integral case, $S^n_k$ means $S^{C_n}_{C_k}$.

We continue to use the notation $\mz{x}\defeq\max(x,0)$ from \cite{SulSlopes}. We also make use of the Iverson bracket: for a proposition $P$,
\[
   [P] \defeq
   \begin{cases}
     1 & P\text{ is true}\\
     0 & P\text{ is false}
   \end{cases}
\]

\subsection{Overview}
We introduce Tambara functors and establish notation in \sec\ref{sec:background}. Our construction for prisms is given in \sec\ref{sec:main}. We study the ``perfect sandwich'' in \sec\ref{sec:drw}. Finally, we give the construction for $q$-pd thickenings/GNS's in \sec\ref{sec:gns}.

\subsection{Acknowledgments}
The author would like to thank Ben Antieau, Andrew Blumberg, Shachar Carmeli, Alex Frederick, Zhouhang Mao, Maxime Ramzi, and Noah Riggenbach for interesting and helpful conversations related to this work. We also thank Bhargav Bhatt for sharing an example of a non-orientable prism, Sanath Devalapurkar for answering questions about \cite{GNS}, and Jay Shah for comments on v1 of this paper. We especially thank Andrew Blumberg for encouraging us to split this project into two papers.

%% file: background.tex
\section{Tambara functors}
\label{sec:background}

In this section we explain the notion of Tambara functors. We assume the reader is familiar with Mackey functors and prisms as covered in \cite[Part 1]{SulSliceTHH}. Our goal here is just to give the intuition and review the facts we need. A full treatment is given in \cite{HillHandbook} and \cite{IncompleteTambara}. We also recommend, and will rely on, \cite{MazurTambara} for an explicit description in the case of cyclic groups.

We informally define Tambara functors and survey relevant examples in \sec\ref{sub:idea-examples}. We give a more precise account of the axioms in \sec\ref{sub:axioms}.

\subsection{Idea and examples}
\label{sub:idea-examples}

If $A$ is a ring with $G$-action, we may consider the fixed-rings $A^H$, as $H$ ranges over subgroups of $G$. There is a great deal of structure relating these: if $K\le H$, then we have

\begin{itemize}
   \item a residual action of the Weyl group $W_H(K) = \frac{N_H(K)}K = \Aut_G(H/K)$ on $A^K$;
   \item a restriction map $\res^H_K\colon A^H \to A^K$ given by inclusion of fixed points;
   \item a transfer map $\tr^H_K\colon A^K \to A^H$ given by summing over conjugates, $\tr^H_K(x) = H/K \cdot x$;
   \item a norm map $N^H_K\colon A^K \to A^H$ given by multiplying over conjugates, $N^H_K(x) = x^{H/K}$.
\end{itemize}

There are also isomorphisms $g\colon A^H \iso A^{gHg^{-1}}$, so some version of the above also works when $K$ is merely subconjugate to $H$ (which we denote by $K\subconj H$). We will mostly ignore these in our discussion since ultimately we are only interested in $C_n$.

A $G$-Tambara functor is a collection of rings $A^H$ indexed by the (conjugacy classes of) subgroups of $G$, along with maps $\res^H_K$, $\tr^H_K$, $N^H_K$ connecting them in the same way as above. We will give more precise axioms as we need them; the important parts are
\begin{align*}
   \res^H_K \tr^H_K(x) &= H/K\cdot x\\
   \res^H_K N^H_K(x) &= x^{H/K}
\end{align*}
and that
\begin{align*}
   N^H_K\colon A^K \to{} &A^K/\P_K\\
   &\defeq A^K/\{\im\tr^K_J \mid J\le K\}
\end{align*}
is a ring homomorphism. Just as Mackey functors play the role of abelian groups in equivariant algebra, Tambara functors play the role of commutative rings. Mackey functors and Tambara functors are denoted with underlines, e.g.\ $\m A$ and $\m M$. As mentioned in \sec\ref{sub:notation}, we will often write $F^H_K$ instead of $\res^H_K$ and $V^H_K$ instead of $\tr^H_K$.

In the literature this data is often depicted by a Lewis diagram as follows:
\[\xymatrix{
   A^H \ar[d]_-{\res}\\
   A^K \ar@/^2em/[u]^-{N^H_K} \ar@/_2em/[u]_-{\tr^H_K} \ar@(dr,dl)[]^-{W_H(K)}
}\]

Since this gets too busy (and our Weyl actions will all be trivial) we will just write
\[\xymatrix{
   A^H \ar@{-}[d]\\
   A^K
}\]

\begin{warning}
   Many sources in the literature call the map
   \[ X_{hC_p} \to X^{hC_p} \]
   the ``norm'', but it is really a version of the \emph{trace}. (Here $X$ is a spectrum with $C_p$-action.)
\end{warning}

In our case, $G$ is the ind-finite group $\Cpoo=\Q_p/\Z_p$ or $\bigcup C_n = \Q/\Z$. We will mainly be interested in Tambara functors in which all restriction maps are surjective (which forces the Weyl actions to be trivial). These are very far from the motivating example coming from rings with $G$-action; more relevant examples for us are:

\begin{example}[Witt vectors]
   \label{ex:witt}
   If $A$ is a ring with $C_{p^2}$-action, we may define Witt vectors $W_2(A)$ via the ghost polynomials
   \begin{align*}
      w_0 &= a_0\\
      w_1 &= a_0^{C_p/e} + C_p/e\cdot a_1\\
      w_2 &= a_0^{C_{p^2}/e} + C_{p^2}/C_p\cdot a_1^{C_p/e} + C_{p^2}/e \cdot a_2
   \end{align*}
   When the action on $A$ is trivial, these become the usual ghost polynomials and recover the usual truncated Witt vectors.

   We can do something similar for any (pro)finite group $G$; we denote the resulting ring by $\W_G(A)$. For the profinite groups $\Z_p$, resp.\ $\hat\Z$, this recovers $W(A)$, resp.\ $\W(A)$, in the usual sense (giving $A$ the trivial action) \cite{DSBurnside}. In this paper, we will only be concerned with trivial Weyl actions; however, it is important that the theory naturally has this level of generality.

   As $H$ varies over subgroups of $G$, the rings $\W_H(A)$ form a $G$-Tambara functor, which we denote by $\m\W_G(A)$. When $G=\Z_p$ or $\hat\Z$, we denote this simply by $\m W(A)$ or $\m\W(A)$. When $A$ has trivial action, the formulas for $\m W(A)$ are
   \begin{align*}
      F(w_0,\dotsc,w_n) &= (w_1,\dotsc,w_n)\\
      V(w_0,\dotsc,w_{n-1}) &= (0, pw_0, \dotsc, pw_{n-1})\\
      N(w_0,\dotsc,w_{n-1}) &= (w_0, w_0^p, \dotsc, w_{n-1}^p)
   \end{align*}
   Integrally, when $A$ has trivial action, the formulas for $\m\W(A)$ are
   \begin{align*}
      F^{mn}_n(w)_{k\mid n} &= w_{mk}\\
      V^{mn}_n(w)_{k\mid mn} &= [m\mid k]w_{k/m}\\
      N^{mn}_n(w)_{k\mid mn} &= w_{k/g}^g, \quad g=\gcd(m,k)
   \end{align*}

   Brun \cite{BrunWitt} showed that $\m\W_G(A)$ is the \emph{free} Tambara functor on a ring with $G$-action: that is, $\m\W_G$ is the left adjoint
   \[\adjnctn{\CRing^{BG}}{\Tamb_G}{\m\W_G}{} \]
   For $G=C_{p^2}$, this amounts to the fact that in Witt coordinates,
   \[ (a_0, a_1, a_2) = N^2 a_0 + VN a_1 + V^2 a_2. \]
   Note that $\m\W_G(A)$ is \emph{not} the free Tambara functor on the (plain) ring $A$, which would be $\m\W_G(A^{\tnsr G})$.

   In particular, $\m\W_G(\Z)$ is the free Tambara functor on $\Z$ (even as a plain ring), and thus the initial Tambara functor as well as the tensor unit. The ring $\W_G(\Z)$ is known to agree with the \emph{Burnside ring} of finite $G$-sets \cite{DSBurnside}. For example, we have
   \[ C_p/e \times C_p/e = p\cdot C_p/e \]
   and thus
   \[ \W_{C_p}(\Z) = W_1(\Z) = \Z[x]/(x^2-px). \]
\end{example}

\begin{example}[Representation rings]
  For any finite group $G$, we have the representation ring $\RU(G)$ of complex virtual representations of $G$. A basis of this ring is given by the irreducible representations of $G$. For example, for finite cyclic groups we have $\RU(C_n)=\Z[q]/(q^n-1)$.

  As $H$ varies, the rings $\RU(H)$ assemble into a Tambara functor $\m{\RU}$. (If we restrict $H$ to subgroups of $G$, then this is a $G$-Tambara functor; but it is most naturally a ``global'' Tambara functor \cite{GlobalHomotopy}.) The norm and transfers are given by
   \begin{align*}
      \tr^H_K(V) &= V^{\oplus H/K}\\
      N^H_K(V) &= V^{\tnsr H/K}
   \end{align*}
\end{example}

These examples relate to our constructions as follows. When $A=\Ainf(R)$ is a perfect prism, our associated Tambara functor $\m A$ is simply $\m W(R)$. Our second construction $\m D^-$, built from the quotients $D/(q^n-1)$ of the ring $D=\hat\Z\psr{q-1}$, is the profinite completion of $\m{\RU}$.

\subsection{Axioms}
\label{sub:axioms}
We now explain how to generate the ``Tambara reciprocity'' identities which dictate how to move a norm past a sum or transfer. Let $\gamma$ denote a generator of $C_n$. To lighten notation, we adopt the convention
\[ \overline{a_0\dotsc a_{n-1}} = \prod_{i=0}^{n-1} a_i^{\gamma^i} \]
For example, $\overline{x^2y} = x^{\gamma^0} x^{\gamma^1} y^{\gamma^2}$ and $\overline{xyx} = x^{\gamma^0} y^{\gamma^1} x^{\gamma^2}$.

Suppose $K\le H$. Officially, the identities for $N^H_K(x+y)$ come from studying the orbit decomposition of the $G$-set
\[ \P(H/K) = \Hom(H/K, *\amalg *), \]
with a similar story for $N^H_K \tr^K_J$. We show how to do this using elementary formal manipulations.

\begin{example}
   We formally have
   \begin{align*}
      (C_2/e\cdot f)^{C_4/C_2}
         &= (f^{\gamma^0} + f^{\gamma^2})^{\gamma^0 + \gamma^1}\\
         &= (f^{\gamma^0} + f^{\gamma^2})(f^{\gamma^1} + f^{\gamma^3})\\
         &= \sum
         \begin{pmatrix}
            f^{\gamma^0} f^{\gamma^1} & f^{\gamma^0} f^{\gamma^3}\\
            f^{\gamma^2} f^{\gamma^1} & f^{\gamma^2} f^{\gamma^1} 
         \end{pmatrix}\\
         &= C_4/e\cdot(f^{\gamma^0} f^{\gamma^1})
   \end{align*}
   Thus, for a general $C_4$-Tambara functor, the Tambara reciprocity identity for transfers is
   \[ N^4_2 V^2_1(f) = V^4_1(\overline{f^2}). \]
\end{example}

\begin{example}
   We formally have
   \begin{align*}
      (x+y)^{C_4/e}
         &= (x+y)^{\gamma^0 + \gamma^1 + \gamma^2 + \gamma^3}\\
         &= \prod_{i=0}^3 (x^{\gamma^i} + y^{\gamma^i})\\
         &= \sum\begin{pmatrix}
               \ol{x^4} & \blue{\ol{x^3y}} & \blue{\ol{x^2yx}} & \red{\ol{x^2y^2}}\\[.5em]
               \blue{\ol{xyx^2}} & \ol{xyxy} & \red{\ol{xy^2x}} & \green{\ol{xy^3}}\\[.5em]
               \blue{\ol{yx^3}} & \red{\ol{yx^2y}} & \ol{yxyx} & \green{\ol{yxy^2}}\\[.5em]
               \red{\ol{y^2x^2}} & \green{\ol{y^2xy}} & \green{\ol{y^3x}} & \ol{y^4}
            \end{pmatrix}\\
         &= x^{C_4/e} + y^{C_4/e} + C_4/C_2\cdot(\ol{xy})^{C_2/e} + C_4/e\cdot(\blue{\ol{x^3y}}+\red{\ol{x^2y^2}}+\green{\ol{xy^3}})
   \end{align*}
   Thus, for a general $C_4$-Tambara functor, we have
   \[ N^4_1(x+y) = N^4_1(x) + V^4_1(\ol{x^3y}) + [V^4_1(\ol{x^2y^2}) + V^4_2N^2_1(\ol{xy})] + V^4_1(\ol{xy^3}) + N^4_1(y). \]
   Note that the usual binomial coefficient $6=\binom42$ gets broken up as $6=4+2$ here.
\end{example}

\begin{example}
   We formally have
   \begin{align*}
      (C_2/e\cdot f)^{C_6/C_2}
         &= (f^{\gamma^0} + f^{\gamma^3})^{\gamma^0 + \gamma^1 + \gamma^2}\\
         &= (f^{\gamma^0} + f^{\gamma^3})(f^{\gamma^1} + f^{\gamma^4})(f^{\gamma^2} + f^{\gamma^5})\\
         &= \sum
         \begin{pmatrix}
            \blue{f^{\gamma^0} f^{\gamma^1} f^{\gamma^2}} & \blue{f^{\gamma^0} f^{\gamma^1} f^{\gamma^5}}\\
            \red{f^{\gamma^0} f^{\gamma^4} f^{\gamma^2}} & \blue{f^{\gamma^0} f^{\gamma^4} f^{\gamma^5}}\\
            \blue{f^{\gamma^3} f^{\gamma^1} f^{\gamma^2}} & \red{f^{\gamma^3} f^{\gamma^1} f^{\gamma^5}}\\
            \blue{f^{\gamma^3} f^{\gamma^4} f^{\gamma^2}} & \blue{f^{\gamma^3} f^{\gamma^4} f^{\gamma^5}}
         \end{pmatrix}\\
         % &=  \sum \begin{pmatrix}
         %    f^{\gamma^0} f^{\gamma^1} f^{\gamma^2} & f^{\gamma^0} f^{\gamma^1} f^{\gamma^5}\\
         %    & f^{\gamma^0} f^{\gamma^4} f^{\gamma^5}\\
         %    f^{\gamma^3} f^{\gamma^1} f^{\gamma^2} & \\
         %    f^{\gamma^3} f^{\gamma^4} f^{\gamma^2} & f^{\gamma^3} f^{\gamma^4} f^{\gamma^5}
         % \end{pmatrix} + C_6/C_3 \cdot (f^{\gamma^0} f^{\gamma^2} f^{\gamma^4})\\
         &= \blue{C_6/e\cdot(f^{\gamma^0} f^{\gamma^1} f^{\gamma^2})} + \red{C_6/C_3 \cdot (f^{\gamma^0} f^{\gamma^2} f^{\gamma^4})}
   \end{align*}
   Thus, for a general $C_6$-Tambara functor, we have the identity
   \[ N^6_2 V^2_1(f) = V^6_1(\overline{f^3}) + V^6_3 N^3_1(f). \]
\end{example}

In general, we have:

\begin{theorem}[{\cite[Theorems 2.4 and 2.5]{MazurTambara}}]
   Let $G$ be a finite abelian group. There is a set of words $\sfS$ in the letters $x$, $y$, and their conjugates such that
   \[ N^G_K(x+y) = \sum_{K\le H\le G} V^G_H \sum_{\omega\in\sfS} N^H_K(\omega). \]
   In our $p$-typical case, this becomes
   \[ N^m_k(x+y) = \sum_{k\le n\le m} V^m_n \sum_{\omega\in\sfS} N^n_k(\omega). \]
   In our integral case, it becomes
   \[ N^m_k(x+y) = \sum_{k\mid n\mid m} V^m_n \sum_{\omega\in\sfS} N^n_k(\omega). \]
\end{theorem}

\begin{theorem}[{\cite[Theorems 2.8 and 2.9]{MazurTambara}}]
   Let $G$ be a finite abelian group. There is a set of words $\sfT$ in the conjugates of the letter $x$ such that
   \[ N^G_H V^H_K(x) = \sum_{K=H\cap J} V^G_J \sum_{\omega\in\sfT} N^J_K(\omega). \]
   In our $p$-typical case, this becomes
   \[ N^m_n V^n_k(x) = V^m_k \sum_{\omega\in\sfT} \omega. \]
   In our integral case, it becomes
   \[ N^m_n V^n_k (x) = \sum_{\substack{j\mid m\\k=\gcd(n,j)}} V^m_j \sum_{\omega\in\sfT} N^j_k(\omega). \]
\end{theorem}

Remarkably, we will be able to verify these identities without knowing the exact values of $\sfS$ and $\sfT$: it suffices to know that they are the same for every $G$-Tambara functor. This is the power of transversality.

\begin{remark}
   Since our Weyl actions are trivial, in the $p$-typical case, we only need to establish the identities
   \[ FVx = px,\quad FNx = x^p,\quad NVx = p^{p-2}V^2 x^p, \]
   and
   \[ N(x+y) = N(x) + N(y) + \frac Vp\sum_{i=1}^{p-1} \binom pi x^{p-i} y^p. \]
   (We are omitting other identities like $F(x+y)=F(x)+F(y)$ which are trivial to verify.)
\end{remark}

%% file: main.tex
\section{Prisms to Tambara functors}
\label{sec:main}

In this section, we present our main construction. In \sec\ref{sub:main}, we functorially associate a Tambara functor $\m A$ to every prism $(A, I)$. In \sec\ref{sub:subcats}, we identify the images of various special types of prisms under this functor. %Finally, in \sec\ref{sub:thh}, we justify why this construction describes $\mpi_0\TCmin$ (in appropriate circumstances).

\subsection{Main construction}
\label{sub:main}

We give a few descriptions of the construction. Recall that
\[ I_n \defeq I\phi(I)\dotsm\phi^n(I)A. \]

On objects, we set
\[ A^{C_{p^n}} = A/I_n. \]
The restriction maps are the natural quotient maps. To define the transfers and norms, recall that there exist elements $\pi_n\in\phi^n(I)A$ with $\pi_n \equiv p\bmod I_{n-1}$ \cite[Proposition 3.25]{SulSliceTHH}, \cite[Lemma 2.2.8]{APC}. We define
\begin{align*}
  V, N\colon A/I_{n-1} &\to A/I_n\\
  V(x) &= \pi_n x\\
  N(x) &= \phi(x) - \pi_n\delta(x)
\end{align*}
Observe that $FVx = x^p$, $FNx = x^p$, and that $N$ is a ring homomorphism (namely $\phi$) mod $V$. When $I=(d)$ is principal, we may take $\pi_n = u_n \phi^n(d)$ for some unit $u_n\in A^\times$.

\begin{figure}[h]
  \[\xymatrix{
    A/I_2 \ar@{-}[d] \ar[r] & A/\phi(I_1) \ar@{-}[d] \ar[r] & A/\phi^2(I)\\
    A/I_1 \ar@{-}[d] \ar[r] & A/\phi(I)\\
    A/I
  }\]
  \caption{The Tambara functor $\m A$ and its quotients by transfers.}
\end{figure}

To make this construction independent of choices, we will Kan extend from the case of transversal prisms. Recall that for a transversal prism, we have an injection
\[ c_\trans\colon A/I_n \mono \prod_{i=0}^n A/\phi^i(I) \]
by \cite[Lemma 3.7]{AClBTrace}. We will call these \emph{transversal coordinates} and write them as $(t_0,\dotsc,t_{n})$. To work effectively with transversal coordinates, we need to characterize the image of $A/I_n$.

\begin{lemma}
\label{lem:trans-coords}
A vector $(t_0,\dotsc,t_n)$ is in the image of $c_\trans$ if and only if $(t_0,\dotsc,t_{n-1})=c_\trans(t)$ is in the image of $c_\trans$, and $t=t_n$ in $A/(I_{n-1},\phi^n(I))=A/(I_{n-1},p)$.
\end{lemma}

\begin{remark}
For any prism $(A, I)$, there is a comparison map
\[ W_m(A/I) \to A/I_m. \]
If we use ghost coordinates on the source and transversal coordinates on the target, this map is given by
\[ (w_0,\dotsc,w_m) \mapsto (w_m,\phi(w_{m-1}),\dotsc,\phi^m(w_0)). \]
We will construct this map and derive this formula in \sec\ref{sec:drw}, but seeing it now is helpful for understanding the ``backward'' formulas in Construction \ref{cons:main}.
\end{remark}

We are ready to give the official definition of $\m A$.

\begin{construction}
\label{cons:main}
Let $A$ be a transversal prism. We define a Tambara functor $\m A$ as follows. On objects,
\[ \m A(G/C_{p^n}) = A/I_n. \]

We define the Tambara structure maps using transversal coordinates as
\begin{align*}
  F(t_0,\dotsc,t_n) &= (t_0,\dotsc,t_{n-1})\\
  V(t_0,\dotsc,t_{n-1}) &= (pt_0,\dotsc,pt_{n-1},0)\\
  N(t_0,\dotsc,t_{n-1}) &= (t_0^p,\dotsc,t_{n-1}^p,\phi(t_{n-1}))
\end{align*}
We need to verify that this definition satisfies the conditions of Lemma \ref{lem:trans-coords}. This is clear for $F$. If $(t_0,\dotsc,t_{n-1})=c_\trans(t)$, then $(pt_0,\dotsc,pt_{n-1})=c_\trans(pt)$ and $(t_0^p,\dotsc,t_{n-1}^p)=c_\trans(t^p)$. Finally, we have $pt_{n-1}\equiv0$ and $\phi(t_{n-1})\equiv t_{n-1}^p$ mod $I_{n-1}+\phi^n(I)A$ since $p\in I_{n-1}+\phi^n(I)A$ by the extended prism condition.

From the description in transversal coordinates, we see that every identity between $F,V,N$ that holds in the Witt Tambara functor $\m W$, for example
\[ NVx=p^{p-2}V^2x^p, \]
is also satisfied in $\m A$. We have thus indeed defined a Tambara functor (if this does not satisfy the reader, we give a ``fully coherent'' proof of an analogous fact in \sec\ref{sec:gns}, which can be $p$-typified to cover the present situation). We now animate this construction \cite[5.1.4]{CSPurity}, associating an \emph{animated} Tambara functor $\m A$ to every animated prism $A$.

It remains to show that $\m A$ is a static Tambara functor when $A$ is a static (not necessarily transversal) prism. For this, we must show that for any transversal prism $(B,J)$ mapping to $(A,I)$, the map
\[ A\tnsr_B B/J_n \to A/I_n \]
is an isomorphism. Note that $\pi_1(A\tnsr_B B/J_n)=\Tor^B_1(A,B/J_n)$ is the $J_n$-torsion in $A$. Since $I_n$ is the pushforward of $J_n$ \cite[Lemma 3.5]{Prismatic}, and $A$ is $I_n$-torsionfree by definition of prism, this Tor term is zero. The same lemma shows that the map is an isomorphism on $\pi_0$.
\end{construction}

\begin{warning}
In the literature, ``derived Mackey functors'' refers to something different \cite{KaledinMackey,SpecDerivedMackey,AMGRMackey}. We really want animated $\m W(\Z)$-modules.
\end{warning}

\begin{remark}
For perfect prisms, we additionally have an $R$ map given by
\[ R(t_0,\dotsc,t_n) = (\phi^{-1}(t_1),\dotsc,\phi^{-1}(t_n)) \]
in transversal coordinates. This map does not exist for general prisms (although the quotient map $A/I_n \to A/\phi(I_{n-1})$ is nearly as good).
\end{remark}

Recall from above that the norm $N^{n+1}_n(f\bmod I_n)$ is lifted by the expression $\~N^{n+1}_n(f) = \phi(f)-\pi_{n+1}\delta(f)$. It is useful to have a similar explicit expression for iterates of the norm.

\begin{definition}
  In a $\delta$-ring, the operations $\theta_n$ are given by
  \[ \phi(f^{p^{n-1}}) = f^{p^n} + p^n\theta_n(f). \]
  In particular, $\theta_1=\delta$.
\end{definition}

\begin{proposition}
  In a $\delta$-ring, we have
  \[ \phi^n(f) = f^{p^n} + \sum_{i=1}^n p^i \theta_i(f)^{\phi^{n-i}}. \]
\end{proposition}

\begin{lemma}
  \label{lem:n-lift-p}
  The expressions
  % \[
  %   \~N^{m+n}_n(f) = f^{\phi^m} - \sum_{i=1}^m \theta_i(f)^{\phi^{m-i}} \prod_{j=0}^{i-1} \pi_{m+n-j}
  % \]
  \begin{align*}
    \~V^{m+n}_n(f) &= \pi_{n+1}\dotsm \pi_{m+n}\\
    \~N^{m+n}_n(f) &= f^{\phi^m} - \sum_{i=1}^m \~V^{m+n}_{m+n-i}\theta_i(f)^{\phi^{m-i}} 
  \end{align*}
  are lifts of $V^{m+n}_n(f\bmod I_n)$ and $N^{m+n}_n(f\bmod I_n)$.
\end{lemma}
\begin{proof}
  The claim for $\~V^{m+n}_n$ is clear. The norm $N^{m+n}_n(f\bmod I_n)$ is characterized in transversal coordinates by
  \[
    N^{m+n}_n(f\bmod I_n) =
    \begin{cases}
      f^{p^m} & {}\bmod {I_n}\\
      \phi^k(f)^{p^{m-k}} & {}\bmod{\phi^{n+k}(I)},\quad 1\le k\le m
    \end{cases}
  \]
  so we just need to check that $\~N^{m+n}_n$ satisfies this congruence. For any $0\le k\le m$, we have
  \begin{align*}
    \~N^{m+n}_n(f)
      &= f^{\phi^m} - \sum_{i=1}^m\~V^{m+n}_{m+n-i} \theta_i(f)^{\phi^{m-i}}\\
      &\equiv f^{\phi^m} - \sum_{i=1}^{m-k} p^i\theta_i(f)^{\phi^{m-i}} \quad\bmod {\phi^{n+k}(I),\text{ or }I_n\text{ if }k=0}\\
      &= \phi^k\left(f^{\phi^{m-k}}-\sum_{i=1}^{m-k} p^i\theta_i(f)^{m-k-i}\right)\\
      &= \phi^k(f^{p^{m-k}})\\
      &= \phi^k(f)^{p^{m-k}}.\qedhere
  \end{align*}
\end{proof}

\begin{remark}
  Similarly, one can check that for any ring $R$, we have
  \[ N^n(f) = f - \sum_{i=1}^n V^i \theta_i(f) \]
  in $W(R)$.
\end{remark}

\subsection{Subcategories}
\label{sub:subcats}

We now identify the image of transversal, crystalline, and perfect prisms under the functor constructed above. The point is to give definitions that make sense for every $G$, so that these notions will make sense for $G$-prisms once those have been defined. The fact that we are able to do so is evidence for thinking that ``$G$-prisms'' might exist.

% \begin{definition}
%   A Mackey functor $\m M$ is \emph{bounded} if $M^K$ has bounded $T_{H/K}^\infty$-torsion for all $K\subconj H\le G$.
% \end{definition}

% \begin{remark}
%   In our context, $\m A$ being bounded is equivalent to each $A/I_n$ having bounded $p^\infty$-torsion. We use this definition because $p$ arises as $\res\circ\tr$ in some of the proofs in \cite{PrismsTambara2}.
% \end{remark}

% \begin{proposition}
%   A prism $A$ is bounded if and only if $\m A$ is a bounded Mackey functor.
% \end{proposition}
% \begin{proof}
%   We need to show that when $A/I$ has bounded $p^\infty$-torsion, so does each $A/I_n$. We imitate the proof of \cite[Lemmas 2.2.\{1,5\}]{APC}. First note that $A$-modules with bounded $p^\infty$-torsion are closed under extension.
% \end{proof}

\begin{definition}
  A Mackey functor $\m M$ is \emph{transversal} if the natural map
  \[ M^H \to \prod_{K\subconj H} \H^0(W_H(K), M^K)/\P_K \]
  is an injection for all $H\le G$.
\end{definition}

\begin{remark}
  Let $\m M$ be a Mackey functor, and let $\cF$ be a family of subgroups of $G$. Viewing $\m M$ as a $G$-spectrum, there is a pullback square
  \[\xymatrix{
    \m M \ar[r] \ar[d] \ar@{}[dr]|\square & \~{E\cF} \tnsr \m M \ar[d]\\
    \m M^{E\cF_+} \ar[r] & \~{E\cF} \tnsr \m M^{E\cF_+}
  }\]
  We expect that $\m M$ being transversal is equivalent to $\mpi_0$ of this square being a pullback for all families $\cF$. Indeed, this was our original definition of transversality. We have adopted the above definition for expediency, since we want to postpone developing the theory of transversality to future work. Notably, for transversal Mackey functors, there should be a version of the stratified perspective on equivariance \cite{Glasman,AMGRStratified} that stays in the world of Mackey functors.
\end{remark}
 
% \begin{remark}
%   If the subgroups of $G$ are linearly ordered, then transversality simplifies to the condition that
%   \[\xymatrix{
%     M^H \ar[r] \ar[d] & M^H/\tr M^K \ar[d]\\
%     \H^0(W_H(K), M^K) \ar[r] & \hat\H^0(W_H(K), M^K)
%   }\]
%   be a pullback for all $K\subconj H\le G$.
% \end{remark}

\begin{proposition}
A prism $A$ is transversal if and only if $\m A$ is a transversal Mackey functor.
\end{proposition}
\begin{proof}
  The ``only if'' direction is \cite[Lemma 3.7]{AClBTrace}. Conversely, if $\m A$ is a transversal Mackey functor, then in particular
  \[ A/I_1 \hookrightarrow A/I \times A/\phi(I). \]
  Suppose $f\in A$ with $pf\in I$. Since $p\in I+\phi(I)A$, this is equivalent to $\pi_1f\in I$, where $\pi_1$ is a generator of $\phi(I)$. The above injection implies that $I\cap \phi(I)=I_1$, so we have $\pi_1 f\in I_1$. Since $\pi_1$ is a non-zerodivisor, we can cancel to get $f\in I$. This shows that $A/I$ is $p$-torsionfree.
% The condition that $\m A$ be transversal simplifies to the condition that the dotted arrow $f$ 
% \[\xymatrix{
%   A/I_{m+n} \ar[rrd] \ar[ddr] \ar@{..>}[dr]|-f\\
%   & P \pullback \ar[r] \ar[d] & A/\phi^{n+1}(I_{m-1}) \ar[d]\\
%   & A/I_n \ar[r] & A/I_n,p^m
% }\]
% be an isomorphism for all $m, n$, where $P$ is the pullback. The map $f$ is always surjective since $A/(I_n,p^m) = A/(I_n,\phi^{n+1}(I_{m-1}))$. The condition that $f$ be injective is equivalent to $A/I_n$ being $p^m$-torsionfree. This implies transversality by taking $n=0$; the converse is \cite[Lemma 2.2.5]{APC}.
\end{proof}

Next we turn to crystalline prisms. One can easily characterize the image of crystalline prisms as those $\m A$ such that $p=0$ in $A^e$. However, this characterization is not very natural from the equivariant perspective.

\begin{definition}
A Mackey functor $\m M$ is \emph{cohomological} if $\tr^H_K\res^H_K x=|H/K|x$ for all $K\preceq H$ and all $x\in M^H$. A Tambara functor $\m A$ is cohomological if it is cohomological as a Mackey functor, and furthermore $N^H_K\res^H_K x = x^{|H/K|}$ for all $K\preceq H$ and all $x\in A^H$.
\end{definition}

\begin{remark}
  For $G=C_2$, a Tambara functor $\m A$ is cohomological as soon as it satisfies $N^H_K\res^H_K x = x^{|H/K|}$. This follows from the formula
  \[ N(x+y) = N(x) + N(y) + V(x\bar y) \]
  applied to $y=1$ \cite[Example 3.8(i)]{DMPPolynomial}. We do not know if this is true for general $G$.
\end{remark}

\begin{remark}
  Being a cohomological Mackey functor (resp.\ Tambara functor) is the same as being a $\m\Z$-module (resp.\ $\m\Z$-algebra).
\end{remark}

We wanted to prove the following:

\begin{conjecture}
  \label{conj:crystalline}
  A prism $(A, I)$ is crystalline if and only if $\m A$ is a cohomological Tambara functor. As we will see below, this is equivalent to the conjecture that $A$ is crystalline if and only if $\phi(I)A=(p)$.
\end{conjecture}

Unfortunately, we were unable to prove this. We can do so in the case $p=2$, however:

\begin{proposition}
  Conjecture \ref{conj:crystalline} is true when $p=2$.
\end{proposition}
\begin{proof}
  Working locally, we can assume that $I=(\g)$ is principal. Recall that $V\colon A/I\to A/I_1$ can be given by multiplication by $\phi(\g)/\delta(\g)$. Thus, the statement that $\m A$ is cohomological, i.e.\ $V(1)=p$, means that
  \[ \frac{\phi(\g)}{\delta(\g)} \equiv p \bmod I_1 .\]
  (For a general prism, this is only true mod $I$.) This gives
  \begin{align*}
    \frac{\g^p+p\delta(\g)}{\delta(\g)} &= p + \g\phi(\g)x\\
    \frac{\g^p}{\delta(\g)} &= \g\phi(\g)x\\
    \g^{p-1} &= \phi(\g)x\delta(\g)
  \end{align*}
  On the other hand, we can write
  \[ p=\phi(\g)(\delta(\g)^{-1}-\g x) \]
  which shows that $\phi(\g)A=(p)$ by \cite[Lemma 2.24]{Prismatic}. Combining these, we have $p\mid\g^{p-1}$. When $p=2$, this means $p\mid\g$, so $(p)=(\g)$ by another application of \cite[Lemma 2.24]{Prismatic}.
\end{proof}

\begin{remark}
  In one failed attempt to prove Conjecture \ref{conj:crystalline}, we discovered the identity
  \[ \delta(f^{n+1}) = \delta(f)\sum_{i=0}^n \phi(f)^{n-i} f^{pi}. \]
  This can be proved by induction using the identity
  \[ \delta(fg) = \phi(f)\delta(g) + \delta(f)g^p. \]
  Sadly, this identity plays no role in the present paper; but it is a nice identity, so we wanted to share it.
\end{remark}

\begin{theorem}
A prism $(A,I)$ is perfect if and only if the de Rham-Witt comparison map
\[ \m W(A/I) \to \m A \]
constructed in \sec\ref{sec:drw} is an equivalence.
\end{theorem}
\begin{proof}
  The ``only if'' direction is well-known, so suppose that $\m W(A/I)\to \m A$ is an equivalence. We will show that $\phi\colon A\to A$ is an isomorphism. Working locally, we can assume that $I=(d)$ is principal.

  Taking the limit over $F$, we get an isomorphism $W((A/I)^\flat) \iso A$. Since $(A/I)^\flat$ is a perfect $\F_p$-algebra, we can apply \cite[Lemma 2.34]{Prismatic} to get that $\phi(d)$ is a non-zerodivisor.

  Let us write $\phi_0\colon A/I\to A/\phi(I)$ for the map induced by $\phi$, in order to distinguish it from $\phi\colon A\to A$. Note that $\phi_0$ is the composite
  \[\xymatrix@ur{
    A/I \ar[d]_N \ar[dr]^-{\phi_0}\\
    A/I_1 \ar[r]_-R & A/\phi(I)
  }\]
  Since this composite is always an isomorphism for the Witt vectors Tambara functor, we have that $\phi_0$ is an isomorphism.

  Suppose that $\phi(f)=0$. Since $\phi_0$ is injective, we get $f=dg$ for some $g$. But now $0=\phi(f)=\phi(d)\phi(g)$; since $\phi(d)$ is a non-zerodivisor, we get $\phi(g)=0$. Iterating, we get $f\in(\phi(I)A)^\infty$. Since $A$ is $\phi(I)$-adically separated, this shows $f=0$.

  Next let $f\in A$ be arbitrary. Since $\phi_0$ is surjective, we can write $f=\phi(g_0) + \phi(d) h_0$. We can then write $h_0=\phi(g_1)+\phi(d) h_1$. Iterating this process, we have
  \begin{align*}
    f &= \phi(g_0) + \phi(d)h_0\\
      &= \phi(g_0 + dg_1) + \phi(d)^2 h_1\\
      &= \phi(g_0 + dg_1 + d^2g_2) + \phi(d)^3 h_2\\
      &= \dots
  \end{align*}
  Since $A$ is complete with respect to both $I$ and $\phi(I)A$, this converges to a preimage of $f$ for $\phi$.
\end{proof}

%% file: drw.tex
\section{The perfect sandwich}
\label{sec:drw}

In this section, we apply our results to construct natural maps
\[
  W_m(A/I_n) \xra c A/I_{m+n} \xra{\delta_\bullet} W_m(A/\phi^m(I_n)),
\]
and prove some results about them. We refer to this diagram as the ``perfect sandwich'' since the outer terms are related (by taking the limit over $m$) to inverse and direct perfection, respectively. 

\begin{remark}
There is a direct way to construct these maps in the case $m=1,n=0$. For any prism $(A,I)$, we have a commutative diagram
\[\xymatrix{
  W_1(A/I) \ar@{..>}[r]^-c \ar@/^2em/[rr]^-F \ar[d]_-R \pullback & A/I_1 \ar[r] \ar[d] \pullback & A/I \ar[r]^-\phi \ar[d]^-\can & A/\phi(I) \ar[d]\\
  A/I \ar[r]_-\phi & A/\phi(I) \ar[r]_-\can & {\underbrace{A/I,\phi(I)}_{{}=A/I,p}} \ar[r]_-\phi & {\underbrace{A/\phi(I),\phi^2(I)}_{{}=A/\phi(I),p}}
}\]
and the indicated squares are pullbacks when $(A,I)$ is transversal. The map $c$ is induced by the universal property of $A/I_1$; in other words, the point is that the Frobenius $A/I\to (A/I)/p$ factors over $A/\phi(I)$. The rectangle formed by the middle and right squares induces the map $A/I_1 \to W_1(A/\phi(I))$.

Over the Breuil-Kisin prism, this becomes
\[\xymatrix{
  W_1(\O_K) \ar[r] \ar[d]_-R & \BK/I_1 \ar[r] \ar[d] & \O_K \ar[r]^-\can \ar[d] & \O_K[\pi^{1/p}] \ar[d]^-\can\\
  \O_K \ar[r]_-\can & \O_K[\pi^{1/p}] \ar[r]_-{x\mapsto x^p} & \O_K/p \ar[r]_-\can & \O_K[\pi^{1/p}]/p
}\]

Contemplating this diagram, we see that
\[ \BK/I_1 = \{(w_0, w_1) \in W_1(\O_K[\pi^{1/p}]) \mid w_1\in\O_K \}. \]
\end{remark}

\subsection{de Rham-Witt comparison}
\label{sub:dRW}

Typically, the only way to prove anything about Witt vectors is to use ghost coordinates ($w_i$). However, Witt coordinates ($a_i$) are amenable from the Tambara point of view: we have
\[ (a_0,\dotsc,a_n) = \sum_{i=0}^n V^i N^{n-i} (a_i). \]

As a consequence, we can restrict $\m A$ to $C_{p^n}$, extend it back up by Witt vectors on $A/I_n$, and the result will map naturally to $\m A$. We write this as $W^G_n\res^G_n \m A$. Here are two examples of this:
\[
  \vcenter{\xymatrix{
    W_2(A/I) \mt \ar[r] & A/I_2 \mt\\
    W_1(A/I) \mt \ar[r] & A/I_1 \mt\\
    A/I \ar@{=}[r] & A/I
  }}
  \qquad
  \vcenter{\xymatrix{
    W_1(A/I_1) \mt \ar[r] & A/I_2 \mt\\
    A/I_1 \mt \ar@{=}[r] & A/I_1 \mt\\
    A/I \ar@{=}[r] & A/I
  }}
\]
For $n=0$, this is a map $\m W(A/I)\to \m A$, and the components of this map were constructed by Molokov \cite{Molokov}. Our construction characterizes this map by a universal property, and shows that it commutes with $F,V,N$.

\begin{lemma}
  Using ghost coordinates on the source and transversal coordinates on the target, our comparison map
  \[ c\colon W_m(A/I_n) \to A/I_{m+n} \]
  is given by
  \[ c(w_0,\dotsc,w_m) = (\underbrace{w_m,\dotsc,w_m}_{n+1\text{ times}}, \phi(w_{m-1}), \dotsc, \phi^m(w_0)). \]
\end{lemma}
\begin{proof}
  For an element $x\in W_m(A/I_n)$, let us write $a_i(x)$ for its Witt coordinates and $w_i(x)$ for its ghost coordinates. Similarly, we write $t_i(y)$ for the transversal coordinates of an element $y\in A/I_{m+n}$. Note that $\min(j,n) = j-\mz{j-n}$. We have
  \begin{align*}
    t_j(c(x))
      &= t_j \sum_{i=0}^m V^i N^{m-i} a_i(x)\\
      &= \sum_{i=0}^m [j<m+i] p^i t_j (N^{m-i} a_i(x))\\
      &= \sum_{i=0}^m [j<m+n-i] p^i \phi^{\mz{j-n}}(t_{j-\mz{j-n}}(a_i(x))^{p^{m-i-\mz{j-n}}})\\
      \intertext{We will now suppress the $t_{j-\mz{j-n}}$ since it no longer matters. Continuing to rewrite, we get}
      &= \phi^{\mz{j-n}} \sum_{i=0}^{m-\mz{j-n}} p^i a_i(x)^{p^{m-i-\mz{j-n}}}\\
      &= \phi^{\mz{j-n}} w_{m-\mz{j-n}}(x). \qedhere
  \end{align*}
\end{proof}

\begin{corollary}
When $n=0$, our de Rham-Witt comparison map $c$ agrees with that constructed by Molokov.
\end{corollary}
\begin{proof}
Molokov works by reduction to the case of transversal prisms. While this formula does not appear explicitly in \cite{Molokov}, a close reading of his construction shows that it is given by
\[ (w_0,\dotsc,w_n) \mapsto (w_n,\phi(w_{n-1}),\dotsc,\phi^n(w_0)).\qedhere \]
\end{proof}

\begin{remark}
  Molokov uses the map $W_m(A/I)\to A/I_m$ to construct a de Rham-Witt comparison map
  \[
    W_m\Omega^{*,\cont}_{S/(A/I)} \ctnsr^L_{W_m(A/I)} A/I_m
    \to
    H^*(\prism_{S/A}/I_m)
  \]
  and shows this to be an equivalence when $A$ is perfect transversal. Can we do something similar for the maps $W_m(A/I_n)\to A/I_{m+n}$?

  We can try to do so as follows. Fix a base prism $A$, let $G=C_{p^\infty}$, let $H=C_{p^n}$, and let $\m S$ be an $H$-Tambara functor which is a $\res^G_H\m A$-algebra. We can define a relative prismatic site $(\m S/A)_\prism$ using maps $S\to\res^G_H\m B$ for prisms $B$ under $A$:
  \[\xymatrix@R=1.5em{
    A \ar@{-}[dd]|-{\hole\vdots\hole} \ar[rr] && B \ar@{-}[dd]|-{\hole\vdots\hole}\\
    \\
    A/I_2 \ar@{-}[d] && B/I_2\ar@{-}[d]\\
    A/I_1 \ar@{-}[d] \ar[r] & S^{C_p} \ar@{-}[d] \ar[r] & B/I_1 \ar@{-}[d]\\
    A/I \ar[r] & S^e \ar[r] & B/I
  }\]

  With this definition, we obtain natural maps
  \[ W_m(S^{C_{p^n}}) \to B/I_{m+n} \]
  inducing
  \[
    W_m(S^{C_{p^n}})
    \to
    H^0(\prism_{\m S/A}/I_{m+n})
    \hookrightarrow
    H^*(\prism_{\m S/A}/I_{m+n})
  \]
  and thus
  \[
    W_m\Omega^{*,\cont}_{S^H/A^H} \ctnsr^L_{W_m(A^H)} \to H^*(\prism_{\m S/A}/I_{m+n}).
  \]
  This notion of prismatic cohomology might arise as the associated graded of a motivic filtration on the relative $\THH$ defined by Angeltveit-Blumberg-Gerhardt-Hill-Lawson-Mandell \cite{ABGHLM}.

  However, observe that this notion of prismatic cohomology vanishes for the fixed-point Tambara functor
  \[
    \m{\F_{p^2}}
    =
    \left(
    \vcenter{\xymatrix{
      \F_p \ar@{-}[d]\\
      \F_{p^2}
    }}
    \right)
  \]
  with the usual Galois action on $\F_{p^2}$, since $p\notin I_1$ for any nonzero prism. This shows that at least one of the following is true:
  \begin{itemize}
    \item this is the wrong definition of prismatic cohomology of Tambara functors;
    \item $\m{\F_{p^2}}$ is the wrong kind of input for this cohomology theory.
  \end{itemize}
\end{remark}

We can also give a direct formula for $c$ using the lifts of Lemma \ref{lem:n-lift-p}. Recall that these are given by
\begin{align*}
  \~V^{m+n}_n(f) &= \pi_{n+1}\dotsm \pi_{m+n}\\
  \~N^{m+n}_n(f) &= f^{\phi^m} - \sum_{i=1}^m \~V^{m+n}_{m+n-i}\theta_i(f)^{\phi^{m-i}}.
\end{align*}

\begin{lemma}
  \label{lem:c-witt}
  In Witt coordinates, the comparison map
  \[ c\colon W_m(A/I_n) \to A/I_{m+n} \]
  is given by
  \[
    c(a_0,\dotsc,a_m) =
    \sum_{k=0}^m
      V^{m+n}_{m+n-k}\left(
        a_k - \sum_{i=1}^k \theta_i(a_{k-i})
      \right)^{\phi^{m-k}} \]
\end{lemma}
\begin{example}
  For $m=2$, we have
\[
  c(a_0,a_1,a_2)
  =
  \sum\left(\vcenter{\xymatrix@1@=1em{
  a_0^{\phi^2} & \gray{-V^2_1}\theta_1(a_0)^\phi & \gray{-V^2_0}\theta_2(a_0)\\
    & \gray{V^2_1} a_1^\phi & \gray{-V^2_0}\theta_1(a_1)\\
    && \gray{V^2_0} a_2
  }}\right)
\]
\end{example}
\begin{proof}
  We compute
  \begin{align*}
    c(a_0,\dotsc,a_m)
      &=
      \sum_{k=0}^m V^{m+n}_{m+n-k} N^{m+n-k}_n(a_k)\\
      &= \sum_{k=0}^m V^{m+n}_{m+n-k}\left(
        a_k^{\phi^{m-k}} - \sum_{i=1}^{m-k} V^{m+n-k}_{m+n-k-i} \theta_i(f)^{\phi^{m-k-i}}\right)\\
      &= \sum_{k=0}^m V^{m+n}_{m+n-k}\left(
        a_k^{\phi^{m-k}} - \sum_{i=1}^k \theta_i(a_{k-i})^{\phi^{m-k}}\right)\\
        &= \sum_{k=0}^m V^{m+n}_{m+n-k}\left(
          a_k - \sum_{i=1}^k \theta_i(a_{k-i})\right)^{\phi^{m-k}}\qedhere
\end{align*}
\end{proof}

This can be used to give a cute characterization of the image of $c$.

\begin{definition}
  The \emph{prismatic ghost polynomials} are
  \begin{align*}
    v_0^{(n)} &= b_0\\
    v_1^{(n)} &= b_0^\phi + V^{n+1}_nb_1\\
    v_2^{(n)} &= b_0^{\phi^2} + V^{n+2}_{n+1} b_1^\phi + V^{n+2}_n b_2\\
    &\dots\\
    v_m^{(n)} &= \sum_{k=0}^m V^{m+n}_{m+n-k} b_k^{\phi^{m-k}}
  \end{align*}
\end{definition}

\begin{proposition}
  \label{prop:image-c}
  The image of $W_m(A/I_n)\to A/I_{m+n}$ agrees with the image of the prismatic ghost polynomial $v^{(n)}_m\colon (A/I_n)^{\times m}\to A/I_{m+n}$.
\end{proposition}
\begin{proof}
Set $b_k = a_k - \sum_{i=1}^k \theta_i(a_{k-i})$ in the statement of Lemma \ref{lem:c-witt}.
\end{proof}

\begin{remark}
  The inspiration for the prismatic ghost polynomials comes from topology. Suppose we are in the situation where a spherical lift $\S_A$ is available, so that $\m A$ arises as $\mpi_0\TCmin(A/I|\S_A)$ (or is expected to). By \cite[Theorem 3.3]{HMFinite}, we have
  \[ W_m(A/I) = \pi_0\TR^m(A/I). \]
  But the relative and absolute $\TR^m$ have the same $\pi_0$, so we get
  \[ \pi_0\TR^m(A/I|\S_A) = \pi_0\TR^m(A/I) = W_m(A/I). \]
  The map $c$ then describes the map
  \[ W_m(A/I) = \pi_0\TR^m(A/I|\S_A) \to \pi_0\THH(A/I|\S_A)^{hC_{p^m}} = A/I_m. \]

  The prismatic ghost polynomials (for $n=0$) come from the Nikolaus-Scholze ``staircase'' formula for $\TR^m$ \cite[Corollary II.4.7]{NikolausScholze}. For example, when $m=1$ this staircase formula gives a pullback
  \[\xymatrix{
    W_1(A/I) \pullback \ar[r] \ar[d] & A/I \ar[d]^-\phi\\
    A/I_1 \ar[r] & A/\phi(I).
  }\]
  This diagram says that $W_1(A/I)$ is the image of $(v_0^{(0)}, v_1^{(0)})$.
\end{remark}

\subsection{\tops{$\delta$}{delta}-operations}
\label{sub:delta}

Looking at the map $c$ for long enough, one realizes that there is also a map
\[ A/I_{m+n}\to W_m(A/\phi^m(I_n)). \]
We explain how to see this using the $\delta$-structure on $A$.

\begin{lemma}
  The map $\delta_\bullet\colon A\to W(A)$ encoding the $\delta$-structure on $A$ descends to a map
  \[
    A/I_{m+n} \xra{\delta_\bullet} W_m(A/\phi^m(I_n)).
  \]
\end{lemma}
\begin{proof}
  The map $\delta_\bullet$ is given in Witt and ghost coordinates by
  \begin{align*}
    \delta_\bullet(x)
      &= (x,\delta(x),\delta_2(x),\dotsc)_{\text{Witt}}\\
      &= (x,\phi(x),\phi^2(x),\dotsc)_{\mathghost}
  \end{align*}
  As before, it suffices to prove the lemma when $A$ is transversal, allowing us to take advantage of ghost coordinates. For any $0\le k\le m$, we have
  \begin{align*}
    \phi^k(I_{m+n})
      &= \phi^k(I_{m-k-1} \phi^{m-k}(I_n) \phi^{m+n-k+1}(I_{k-1}))\\
      &= \phi^k(I_{m-k-1}) \phi^m(I_n) \phi^{m+n+1}(I_{k-1})\\
      &\subset \phi^m(I_n). \qedhere
  \end{align*}
\end{proof}

\begin{remark}
  In Witt coordinates, the existence of $\delta_\bullet$ tells us, for instance, that
  \[ \delta_k(I_n) \subset \phi^k(I_{n-k}). \]
  We originally constructed $\delta_\bullet$ by proving these identities; but one still needs to reduce to the case of transversal prisms to do so. The commutation with $V$ (Remark \ref{rem:delta-tamb}) can also interpreted this way. However, we do not know of any elementary way to interpret the commutation with $N$.
\end{remark}

We now study the composite $\delta_\bullet c$.

\begin{example}
  For $m=1$, the composite is
  \begin{align*}
    (a_0, a_1)
      &\mapsto Na_0 + Va_1\\
      &= \sum\left(\begin{matrix}
        a_0^\phi & -V\theta_1(a_0)\\
        & Va_1
      \end{matrix}\right)\\
      &= \sum\left(\begin{matrix}
        a_0^p & -V\theta_1(a_0)\\
        p\theta_1(a_0) & Va_1
      \end{matrix}\right)\\
      & \mapsto
      \begin{cases}
        a_0^\phi + V(a_1 - \theta_1(a_0)) &\in A/\phi(I_n) \text{ via } \id\\
        \phi(a_0^p + pa_1) &\in A/\phi(I_n) \text{ via } \phi
      \end{cases}
    \end{align*}
    That is, the composite is given in ghost coordinates by \[ (\phi(w_0),\phi(w_1)) + (V(a_1 - \theta_1(a_0), 0)) = W_1(\phi) + \varepsilon_1 \]
    with $\varepsilon_1=(V(a_1-\theta_1(a_0)), 0)$.
\end{example}

\begin{example}
  For $m=2$, the composite is
  \begin{align*}
    (a_0, a_1, a_2)
      &\mapsto N^2a_0 + VNa_1 + V^2 a_2\\
      &= \sum\left(\begin{matrix}
        a_0^{\phi^2} & -V\theta_1(a_0)^\phi & -V^2 \theta_2(a_0)\\
        & Va_1^\phi & -V^2\theta_1(a_1)\\
        && V^2 a_2
      \end{matrix}\right)\\
    &\mapsto W_2(\phi^2) + \varepsilon_2
  \end{align*}
  where
  \[
    \varepsilon_2 =
    \begin{bmatrix*}[r]
      V\phi(a_1 - \theta_1(a_0)) + V^2(a_2 - \theta_1(a_1) - \theta_2(a_0))\\
      \phi V^2(a_2 - \theta_1(a_1) - \theta_2(a_0))\\
      0
    \end{bmatrix*}
  \]
\end{example}

\begin{theorem}
  \label{thm:epsilon}
  The composite map
  \[ \delta_\bullet c\colon W_m(A/I_n) \to W_m(A/\phi^m(I_n)) \]
  is given in ghost coordinates by $W_m(\phi^m) + \varepsilon_m$, where
  \begin{align*}
    \varepsilon_m(a_0,\dotsc,a_m)_{0\le i\le m}
    &=
    \sum_{k=i+1}^m p^i V^{m+n}_{m+n-(k-i)}\left(a_k-\sum_{j=1}^k \theta_j(a_{k-j})\right)^{\phi^{m-k+i}}
  \end{align*}
\end{theorem}
\begin{proof}
  Recall again that $\min(i,k)=k-\mz{k-i}$. By Lemma \ref{lem:c-witt}, we have
  \begin{align*}
    c(a_0,\dotsc,a_m)
      &= \sum_{k=0}^m
        V^{m+n}_{m+n-k}\left(
          a_k - \sum_{j=1}^k \theta_j(a_{k-j})
        \right)^{\phi^{m-k}}
  \end{align*}
  Since
  \[ \phi^i V^{m+n}_{m+n-k}(1) \equiv p^{\min(i,k)} V^{m+n}_{m+n-\mz{k-i}}(1) \mod \phi^m(I_n), \]
  applying $\phi^i$ to this gives
  \begin{align*}
    \phi^i c(a_0,\dotsc,a_m)
      &= \sum_{k=0}^m p^{\min(i,k)}
        V^{m+n}_{m+n-\mz{k-i}}\left(
          a_k - \sum_{j=1}^k \theta_j(a_{k-j})
        \right)^{\phi^{m-k+i}}\\
      &=
        \sum_{k=0}^i p^k
          \left(
            a_k - \sum_{j=1}^k \theta_j(a_{k-j})
          \right)^{\phi^{m-k+i}} +
        \sum_{k=i+1}^m p^i
          V^{m+n}_{m+n-(k-i)}\left(
            a_k - \sum_{j=1}^k \theta_j(a_{k-j})
          \right)^{\phi^{m-k+i}}\\
  \end{align*}
  The second summand is what we want, so setting $\ell=k-j$, we rewrite the first summand as
  \begin{align*}
    &\phantom{{}=}
    \sum_{k=0}^i p^k
      \left(
        a_k - \sum_{j=1}^k \theta_j(a_{k-j})
      \right)^{\phi^{m-k+i}}\\
    &=
    \phi^m\left(\sum_{k=0}^i 
      p^k 
        a_k^{\phi^{i-k}} - \sum_{\ell=0}^{i-1} \sum_{j=1}^{i-\ell} p^{\ell+j}\theta_j(a_\ell)^{\phi^{i-(\ell+j)}}\right)\\
    &=
    \gray{\phi^m\left(\sum_{k=0}^i 
      p^k 
        a_k^{\phi^{i-k}} - \sum_{\ell=0}^{\textcolor{black}i} \sum_{j=1}^{i-\ell} p^{\ell+j}\theta_j(a_\ell)^{\phi^{i-(\ell+j)}}\right)}\\
    &=
    \phi^m\sum_{k=0}^i 
      p^k \left(
        a_k^{\phi^{i-k}} - \sum_{j=1}^{i-k} p^j\theta_j(a_k)^{\phi^{(i-k)-j}}\right)\\
    &=
    \phi^m\sum_{k=0}^i 
      p^k a_k^{p^{i-k}}\\
    &= \phi^m(w_i)\qedhere
  \end{align*}
\end{proof}

\begin{lemma}
  \label{lem:delta-theta}
  In a $\delta$-ring, we have the identity
  \[ \delta_n(f) = \sum_{k=0}^{n-1} \theta_{n-k}(\delta_k(f)). \]
\end{lemma}
\begin{proof}
  We have
  \begin{align*}
    \phi^{n-1}(f)
    &= \sum_{k=0}^{n-1} p^k \delta_k^{p^{(n-1) -k}}(f)\\
    \phi^n(f)
      &= \sum_{k=0}^{n-1} p^k \phi\left(\delta_k(f)^{p^{(n-1)-k}}\right)\\
      &= \sum_{k=0}^{n-1} p^k \left(\delta_k(f)^{p^{n-k}} + p^{n-k}\theta_{n-k}(\delta_k(f))\right)\\
      &= \sum_{k=0}^{n-1} p^k \delta_k(f)^{p^{n-k}} + p^n\sum_{k=0}^{n-1} \theta_{n-k}(\delta_k(f)) \qedhere
  \end{align*}
\end{proof}

\begin{theorem}
  \label{thm:epsilon-vanish}
  Let $\varepsilon_m = \delta_\bullet c - W_m(\phi^m)$. Then $\varepsilon_m$ vanishes:
  \begin{enumerate}
    \item \label{enum:e-n0} when $n=0$;
    \item \label{enum:e-last} in the last coordinate;
    \item \label{enum:e-delta} on the image of $A\xra{\delta_\bullet} W(A) \to W_m(A/I_n)$.
  \end{enumerate}
\end{theorem}
\begin{proof}
  These all follow from Theorem \ref{thm:epsilon}. For (\ref{enum:e-n0}), note that each $V^{m+n}_{m+n-(k-i)}(1)=V^m_{m-(k-i)}(1)$ is in $\phi^m(I)$. For (\ref{enum:e-last}), the sum $\sum_{k=m+1}^m$ vanishes. For (\ref{enum:e-delta}), apply Lemma \ref{lem:delta-theta}.
\end{proof}
\begin{corollary}
  The composite
  \[ A/I_{m+n} \morph^c W_m(A/\phi^m(I_n)) \morph^{\delta_\bullet} A/\phi^m(I_{m+n}) \]
  is equal to $\phi^m$.
\end{corollary}

\begin{remark}
  The maps $c$ and $\delta_\bullet$ can be iterated, giving an infinite ladder
  \[\xymatrix{
    W_m(A/I_n) \ar[r] \ar[d]_-{W_m(\phi^m)+\varepsilon_m} & A/I_{m+n} \ar[dl] \ar[d]^-{\phi^m}\\
    W_m(A/\phi^m(I_n)) \ar[r] \ar[d]_-{W_m(\phi^m)+\varepsilon_m} & A/\phi^m(I_{m+n}) \ar[d]^-{\phi^m} \ar[dl]\\
    \dots \ar[r] & \dots
  }\]
\end{remark}

\begin{question}
  What does the perfect sandwich say about $\WCart$ \cite{APC}? Note that the map $\delta_\bullet\colon A\to W(A)$ is used in the construction of the map $\Spf A\to\WCart$.
\end{question}

\begin{remark}
  We do not know how to produce the map $\delta_\bullet$ from topology; it may be related to the identification $\TR^{t\T} = \TP$ \cite[Corollary 10]{TCart}.
\end{remark}

\begin{remark}
  \label{rem:delta-tamb}
  The map $\delta_\bullet$ can be viewed as a map of Tambara functors, for example
  \[\xymatrix{
    A/I_2 \mt \ar[rr]^-{(1,\phi,\phi^2)} && W_2(A/\phi^2(I)) \mt\\
    A/I_1 \mt \ar[r]^-{(1,\phi)} & W_1(A/\phi(I)) \mt \ar[r]^-{W_1(\phi)} & W_1(A/\phi^2(I)) \mt\\
    A/I \ar[r]_-\phi & A/\phi(I) \ar[r]_-\phi & A/\phi^2(I)
  }\]
  This hints at a definition of ``$\delta$-structures on Tambara functors'' that would allow us to recover the $\delta$-structure on $A$ from Tambara-theoretic data. The exact definition is rather subtle to pin down, so we will pursue this in future work.
\end{remark}

%% file: gns.tex
\section{Generalized \tops{$n$}n-series and \tops{$q$}q-pd thickenings}
\label{sec:gns}

In this section we describe a variant of the construction which takes a $q$-pd thickening \cite[Definition 16.2]{Prismatic}, \cite[Definition 3.1]{GLQqCrys} as input instead of a prism. This construction sends a $q$-pd pair $(D,I)$ to the Tambara functor
\[ \m D^-(C_{p^n}) = D/I(p^n)_q. \]
In fact, more generally than $q$-analogues, we will work with \emph{generalized $n$-series} (GNS's) in the sense of Devalapurkar-Misterka \cite{GNS}.

One of our (long-term) goals in this project is to give a unified construction of integral prismatic and $q$-crystalline cohomology (rather than building them one prime at a time and then gluing). To this end, we will construct a $\T$ (i.e.\ $\Q/\Z$, since we only care about the finite subgroups) Tambara functor, rather than a $C_{p^\infty}$ one. To avoid cluttering notation with truncation sets, we will only record the (much harder) integral construction; it is trivial to extract the $p$-typical story.

In the integral case, the formulas for $FV$ and $FN$ are a bit more complicated. Let $a,b,m\in\N$ with $a,b\mid m$, and let $g=\gcd(a,b)$, $\ell=\lcm(a,b)$, so that $g=\frac{ab}{\ell}$ and there is a pullback diagram
\[\xymatrix{
  C_g \pullback \ar[r] \ar[d] & C_b \ar[d]\\
  C_a \ar[r] & C_m.
}\]
In our setting, \cite[Definition 2.11(5)]{MazurTambara} becomes
\begin{align}
  \label{eq:FV} F^m_b V^m_a x &= \frac m\ell V^b_g F^a_g x\\
  \label{eq:FN} F^m_b N^m_a x &= (N^b_g F^a_g x)^{m/\ell}
\end{align}

We review GNS's and construct the associated Green functors in \sec\ref{sub:gns-green}. In \sec\ref{sub:gns-tamb}, we extend this to a Tambara structure in the case of so-called transversal GNS's. We relate our norms $N^{mn}_n(f)$ to existing interpretations of the symbol $f^{(m)_{q^n}}$ in \sec\ref{sub:norm-lifts}.

\subsection{Green structure}
\label{sub:gns-green}

\begin{definition}[{\cite[Definition 2.1.4]{GNS}}]
  Let $D$ be a ring. A \emph{generalized $n$-series} (GNS) over $D$ is a function $s\colon\N\to D$ such that
  \begin{enumerate}
    \item $s(0)=0$;
    \item $s(n)$ is not a zero-divisor for any $n>0$;
    \item $s(n-k) \mid s(n)-s(k)$ for all $n>k>0$.
    % \item \label{enum:gns-green} $s(a+b) = s(a) + s(b) \bmod s(a)s(b)$.
  \end{enumerate}
\end{definition}

% \begin{remark}
%  Condition (\ref{enum:gns-green}) is not part of the definition in \cite{GNS}, but rather is an additional assumption in their statement of the $s$-Lucas theorem \cite[Theorem 2.4.2]{GNS}.
% \end{remark}

\begin{definition}
  A GNS $s$ is \emph{reduced} if $s(1)=1$. If $s$ is any GNS, its \emph{reduction} $\~s$ is the reduced GNS given by
  \[ \~s(n) = \frac{s(n)}{s(1)}. \]
\end{definition}

\begin{definition}[{\cite[Remark 2.4.6]{GNS}}]
  If $s$ is a GNS and $m\in\N$, we define a rescaled GNS $s_m$ by $s_m(n) = s(mn)$.
\end{definition}

The key examples of interest come from formal group laws.

\begin{proposition}[{\cite[Proposition 4.3.4]{GNS}}]
  Let $F$ be a formal group law over a ring $R$, and suppose that the $n$-series $[n]_F(t)\in R\psr t$ is not a zero-divisor for any $n>0$. Then the $n$-series $s(n)=[n]_F(t)$ is a GNS over $R\psr t$. Similarly, the function $\~s(n)=[n]_F(t)/t \mathrel{=:} \<n\>_F(t)$ is a GNS over $R\psr t$.
\end{proposition}

\begin{example}
  The additive formal group $\hat\G_a$ has $n$-series $[n]_{\hat\G_a} = n$. We refer to the GNS $s(n)=n$ as the \emph{additive GNS}.
\end{example}

\begin{example}
  The formal multiplicative group $\hat\G_m$ has $n$-series $[n]_{\hat\G_m}=(1+t)^n-1$. Setting $t=q-1$, we obtain the GNS as $s(n)=q^n-1\in\Z\psr{q-1}$. We refer to this as the \emph{multiplicative GNS}. We have $\~s(n)=(n)_q$ and $\~s_m(n)=(n)_{q^m}$.
\end{example}

\begin{example}[{\cite[Example 4.3.2]{GNS}}]
  The \emph{hyperbolic formal group law} is given by $\hat\bfH(x,y)=\frac{xy}{1+xy}$. It is so-named because $\tanh(a+b)=\hat\bfH(\tanh a,\tanh b)$. We refer to the associated GNS as the \emph{hyperbolic GNS}. Making the same substitution $t=q-1$, one can verify that
  \begin{align*}
    [n]_{\hat\bfH}(t)
      &= \frac{(1+t)^n-(1-t)^n}{(1+t)^n+(1-t)^n} &
    \<n\>_{\hat\bfH}(t)
      &= \frac1t\frac{(1+t)^n-(1-t)^n}{(1+t)^n+(1-t)^n}\\
      &= \frac{q^n-(2-q)^n}{q^n+(2-q)^n} &
      &= \frac1{q-1}\frac{q^n-(2-q)^n}{q^n+(2-q)^n}
  \end{align*}
\end{example}

We have not been able to construct full Green structures in general, and we are not sure if it is possible to do so: it may be necessary to invoke the framework of \emph{bi-incomplete Tambara functors} \cite{BiIncomplete}.

\begin{definition}
  Let $\cF$ be a subset of $\N$ stable under division. A GNS $s$ is \emph{$\cF$-Green} if
  \[ s(m) \equiv \frac m\ell \frac{s(a)s(b)}{s(g)} \bmod s(a)s(b) \]
  for all $m\in\cF$ and all $a,b\mid m$. Here $g=\gcd(a,b)$ and $\ell=\lcm(a,b)=\frac{ab}g$.

  When $\cF=\N$, we say that $s$ is \emph{$\T$-Green}. When $\cF=\bigcup_{p\text{ prime}}p^\N$, we say that $s$ is \emph{$\bfP$-Green}.
\end{definition}

\begin{definition}
  A GNS $s$ is \emph{Lucasian} if
  \[ s(a+b) \equiv s(a) + s(b) \bmod s(a)s(b). \]
  This implies that $\~s$ satisfies the hypotheses of the $s$-Lucas theorem \cite[Theorem 2.4.2]{GNS}, but is a slightly stronger statement.
\end{definition}

\begin{lemma}
  \label{lem:lucas-examples}
  The GNS $[-]_F$ associated to a formal group law $F$ is Lucasian.
\end{lemma}
\begin{proof}
  We have $[a+b]_F = F([a]_F, [b]_F)$. The FGL axioms $F(x,0)=x$ and $F(0,y)=y$ imply that $F(x,y)\equiv x+y\bmod xy$, so the claim follows.
\end{proof}

\begin{lemma}
  \label{lem:lucas-extract}
  If $s$ is Lucasian, then
  \[ \~s_b(a) \equiv a \bmod s(b). \]
\end{lemma}
\begin{proof}
  We induct on $a$, the case $a=0$ being trivial. For the induction step,
  \begin{align*}
    s((a+1)b)
      &\in s(ab) + s(b) + s(ab)s(b)D\\
      &\subset as(b) +s(b)D + s(b) + (as(b) + s(b))s(b)D\\
      &\subset (a+1)s(b) + s(b)^2D\\
    \~s_b(a+1) &\in (a+1) + s(b)D\qedhere
  \end{align*}
\end{proof}

\begin{corollary}
  \label{cor:lucas-ell}
  Suppose that $s$ is Lucasian. To verify that $s$ is $\cF$-Green, it suffices to check the case $m=\ell=\lcm(a,b)$.
\end{corollary}
\begin{proof}
  The assumption $a,b\mid m$ implies $m=c\ell$ for some $c$. By Lemma \ref{lem:lucas-extract}, we have
  \[ s(m) \equiv cs(\ell)\bmod s(\ell)^2D. \]
  Suppose we have
  \[ s(\ell)\equiv \frac{s(a)s(b)}{s(g)} \bmod s(a) s(b). \]
  Since $s(\ell)^2D\subset s(a)s(b)D$, this implies that
  \[ s(m) \equiv c\frac{s(a)s(b)}{s(g)} \bmod s(a)s(b).\qedhere \]
\end{proof}

\begin{corollary}
  Suppose that $s$ is Lucasian. The condition for $s$ to be $\cF$-Green is always satisfied when $a\mid b$ or $b\mid a$.
\end{corollary}
\begin{proof}
  By Corollary \ref{cor:lucas-ell}, we need to check
  \[ \frac{s(\ell)}{s(b)} \equiv \frac{s(a)}{s(g)} \bmod s(a). \]
  These are equal to $1$ when $a\mid b$, and $s(a)/s(b)$ when $b\mid a$.
\end{proof}

\begin{corollary}
  A Lucasian GNS is $\bfP$-Green.
\end{corollary}

\begin{lemma}
  The multiplicative GNS is $\T$-Green.
\end{lemma}
\begin{proof}
  By Lemma \ref{lem:lucas-examples} and Corollary \ref{cor:lucas-ell}, it suffices to show that
  \[ (\ell/b)_{q^b} \equiv (a/g)_{q^g} \bmod (q^a-1). \]
  First note that $\ell/b = a/g$. The result then follows because $b$ and $g$ generate the same subgroup of $\Z/a$.
\end{proof}

\begin{theorem}
  \label{thm:gns-green}
  Let $s$ be an $\cF$-Green GNS. The assignments
  \begin{align*}
    \m D^-(C_n) &= D/s(n)\\
    \m D(C_n) &= D/\~s(n)
  \end{align*}
  naturally extend to incomplete Green functors for the indexing system $\cF$.
\end{theorem}
\begin{proof}
  We define $V^{mn}_n$ as multiplication by $\frac{s(mn)}{s(n)}$. Condition \eqref{eq:FV} is built into the definition of Green GNS.
\end{proof}

\subsection{Tambara structure}
\label{sub:gns-tamb}

To construct norm maps, we need an analogue of transversal coordinates. For notational convenience, we assume from now on that $s$ is $\T$-Green; the results can be stated more carefully in terms of bi-incomplete Tambara functors \cite{BiIncomplete}.

\begin{definition}
  A GNS $s$ on a ring $D$ is \emph{transversal} if $D/s(n)$ is $\N$-torsionfree for all $n$.
\end{definition}

\begin{example}
  The GNS's $s(n)=q^n-1$ and $\~s(n)=(n)_q$ over the ring $\Z\psr{q-1}$ are transversal. The GNS $s(n)=n$ is not transversal.
\end{example}

\begin{proposition}
  If $s$ is transversal, then so are $\~s$ and $s_m$.
\end{proposition}

\begin{definition}[{\cite[Remark 2.4.3]{GNS}}]
  Define $\Phi_n(s)=\prod_{d\mid n}s(d)^{\mu(n/d)}$, where $\mu$ denotes the M\"obius function.
\end{definition}

\begin{lemma}
  If $s$ is a transversal GNS on $D$, then the Mackey functors $\m D^-$ and $\m D$ constructed in Theorem \ref{thm:gns-green} are transversal. Equivalently, $\lcm_{d\mid n}\Phi_d(s)=s(n)$. Equivalently, there is an injection
  \[ D/s(n) \hookrightarrow \prod_{d\mid n} D/\Phi_d(s). \]
\end{lemma}
\begin{proof}
  By induction, it suffices to show that $D/s(mn)\hookrightarrow D/s(m)\times D/\~s_m(n)$. Let $f\in D$ such that $f$ goes to $0$ in $D/s(m)\times D/\~s_m(n)$. Then we can write
  \[ f = s(m)x = \~s_m(n)y. \]
  In $D/s(m)$, we have $\~s_m(n)=n$, so $ny=0$. Since $D/s(m)$ is $\N$-torsionfree, this gives $y\in\~s_m(n)D$; since $s(m)\~s_m(n) = s(mn)$, this gives $f\in s(mn)D$.
\end{proof}

\begin{proposition}
  The transfers $V^{mn}_n$ are given in transversal coordinates by
  \[ V^{mn}_n(w)_{k\mid mn} = [k\mid n] mw_k. \]
\end{proposition}

We will construct norm maps $N^{mn}_n$ by writing $m$ as a product of primes and imitating the formulas from Construction \ref{cons:main} (Figure \ref{fig:norms}).

\begin{figure}
  \newcommand{\six}[4]{\left(\xymatrix@=.75em@ur{#2&#4\\#1&#3}\right)}
  \[
    \six{D/\Phi_1}{D/\Phi_2}{D/\Phi_3}{D/\Phi_6}\colon\quad
    N^6_2\six{w_1}{w_2}{}{}=\six{w_1^3}{w_2^3}{\psi^3(w_1)}{\psi^3(w_2)},\quad
    N^6_1\six{w_1}{}{}{}=\six{w_1^6}{\psi^2(w_1)^3}{\psi^3(w_1)^2}{\psi^6(w_1)}
  \]
  \caption{Examples of the norm in transversal coordinates}
  \label{fig:norms}
\end{figure}

For a semiring $B$, write $B^\bullet$ for the commutative monoid given by $B$ under multiplication.

\begin{definition}
  A \emph{$\lambda$-GNS} on a ring $D$ consists of a GNS $s$ on $D$ and a $\lambda$-structure $\lambda\colon D\to\W(D)$. These are required to be compatible in the sense that
  \begin{align*}
    n &\mapsto (\~s(n), \psi^n)\\
    \intertext{(where $\psi^n$ is the Frobenius lift coming from the $\lambda$-structure) defines a monoid homomorphism}
    \N^\bullet &\to D^\bullet \rtimes \End(D).
  \end{align*}
  Abusively, we typically write a $\lambda$-GNS as $(s,\psi)$.
\end{definition}

\begin{example}
  For $s(n)=(q^n-1)$, the compatibility condition is saying that $(mn)_q = (n)_q (m)_{q^n}$.
\end{example}
  
\begin{construction}
  Let $s$ be a $\lambda$-GNS on a ring $D$. We define a map
  \[ N^{mn}_n\colon\prod_{d\mid n}D/\Phi_d(s)\to\prod_{d\mid mn}D/\Phi_d(s) \]
  by
  \[ N^{mn}_n(w)_{k\mid mn} = \psi^{k/\ell}(w_\ell)^{m\ell/k}\qquad \ell=\gcd(n,k). \]
\end{construction}

We will show momentarily that $N$ descends a map out of $D/s(n)$.

\begin{lemma}
  Our norms behave correctly with respect to composition and restriction.
\end{lemma}
\begin{proof}
  First let us check that norms compose correctly:
  \begin{alignat*}2
    (N^{abc}_{ab} N^{ab}_a(w))_{k\mid abc}
      &= \psi^{k/\ell}(N^{ab}_a(w)_\ell)^{c\ell/k} &\qquad \ell &= \gcd(ab,k)\\
      &= \psi^{k/\ell}(\psi^{\ell/j}(w_j)^{bj/\ell})^{c\ell/k} & j&=\gcd(a,\ell)\\
      &= \psi^{k/j}(w_j)^{bcj/k}\\
    N^{abc}_a(w)_{k\mid abc}
      &= \psi^{k/i}(w_i)^{bci/k} &i&=\gcd(a,k)
  \end{alignat*}
  Note that $j=\gcd(a,\gcd(ab,k))=\gcd(a,k)=i$, so indeed $N^{abc}_{ab} N^{ab}_a = N^{abc}_a$.

  Next we check that $N$ commutes past $F$ correctly, i.e.\ we verify \eqref{eq:FN}. Let $a,b\mid m$, and let $g=\gcd(a,b)$. We have
  \begin{alignat*}2
    F^m_b N^m_a(w)_{k\mid b}
      &= \psi^{k/\ell}(w_\ell)^{\frac{m\ell}{ak}} &\quad \ell&=\gcd(a,k)\\
    N^b_g F^a_g(w)_{k\mid b}
      &= \psi^{k/j}(w_j)^{\frac{bj}{gk}} & j &= \gcd(g, k)
  \end{alignat*}
  Since $k\mid b$ and $g=\gcd(a,b)$, we have $\ell=j$. Moreover, we have
  \[ \frac{m\ell/ak}{bj/gk} = \frac{mg}{ab} = \frac m{\lcm(a,b)} \]
  so that \eqref{eq:FN} indeed holds.
\end{proof}

\begin{lemma}
  Our norms satisfy Tambara reciprocity for sums.
\end{lemma}
\begin{proof}
  We have
  \[ N^{mn}_n(x+y)_{k\mid mn} = \psi^{k/g}(x_g + y_g)^{mg/k}\qquad g=\gcd(n,k). \]
  From \cite[Theorem 2.5]{MazurTambara}, the identity we need to establish is
  \[ N^{mn}_n(x+y) = \sum_{n\mid d\mid mn} V^{mn}_d\sum_{\omega\in\sfS_d} N^d_n(\omega) \]
  for certain subsets $\sfS_H\subset\Z[W_G(H)][x,y]$. The exact values of $\sfS_H$ do not matter; what does matter is that they depend only on $G$ and $H$, so in particular (taking a trivial Tambara functor) we have
  \[ (x+y)^{|G|/|K|} = \sum_{K\le H\le G} \frac{|G|}{|H|} \sum_{\omega\in\sfS_H} \omega^{|H|/|K|}. \]

  We expand
  \begin{align*}
    &\phantom{=}\ \sum_{n\mid d\mid mn} V^{mn}_d\sum_\omega N^d_n(\omega)_{k\mid mn}\\
    &= \sum_{\ell\mid m} V^{mn}_{n\ell}\sum_\omega N^{n\ell}_n(\omega)_k\\
    &= \sum_{\ell\mid m} [k\mid n\ell]\frac m\ell\sum_\omega N^{n\ell}_n(\omega)_k\\
    &= \sum_{\ell\mid m} [k\mid n\ell]\frac m\ell\sum_\omega \psi^{k/g}(\omega_g)^{\ell g/k}\qquad g=\gcd(n,k)
    \intertext{Setting $j=\frac{\ell g}k$, the condition $\ell\mid m$ becomes $j\mid\frac{mg}k$ (and the condition $k\mid n\ell$ means that $j$ is integral). Continuing to rewrite, we get}
    &= \psi^{k/g}\sum_{j\mid\frac{mg}k} \frac{mg}{kj}\sum_\omega \omega_g^j\\
    &= \psi^{k/g}(x_g+y_g)^{mg/k}.
  \end{align*}

  There is a bit of sleight-of-hand here in the summation over $\omega$. To be precise, let us write $\omega\in\sfS_{mn:d:n}$ for the set of words we are summing over at the beginning of the proof, which becomes $\sfS_{mn:\frac{njk}g:n}$ in the second-last line. In order for the final step to be valid, we would need to be summing over $\sfS_{\frac{mg}k:j:\frac gk}$. But $\sfS_{a:b:c}$ only depends on $[a:b:c]\in\mathbf P^2_\Z$, so this is fine.
  
  % \begin{align*}
  %   V^n_d(f)_{k\mid n} &= \frac nd x_k [k\mid d]\\
  %   N^d_1(f)_{k\mid n} &= \psi^k(f)^{d/k}
  % \end{align*}
  % We have
  % \[ \psi^k(x+y)^{n/k} = \sum_{\ell\mid\frac nk} \frac n{kl}\sum_{\~\omega\in \sfC_{n/k\ell}(x,y)} \psi^k(\~\omega)^\ell \]
  % and we want to establish the identity
  % \[ N^n_1(x+y) = \sum_{d\mid n} V^n_d \sum_{\~\omega\in C_{n/d}(x,y)} N^d_1(\~\omega). \]
  % We expand
  % \begin{align*}
  %   N^n_1(x+y)_{k\mid n}
  %     % &= \sum_{d\mid n} V^n_d \sum_{\~\omega\in \sfC_{n/d}(x,y)} N^d_1(\~\omega)_{k\mid n}\\
  %     &= \sum_{d\mid n} \frac nd[k\mid d] \sum_{\~\omega\in \sfC_{n/d}(x,y)} \psi^k(\~\omega)^{d/k}
  %     \intertext{Setting $d=k\ell$, this becomes} 
  %     &= \sum_{\ell\mid\frac nk} \frac n{k\ell} \sum_{\~\omega\in \sfC_{n/k\ell}(x,y)} \psi^k(\~\omega)^\ell
  % \end{align*}
\end{proof}
\begin{lemma}
  Our norms satisfy Tambara reciprocity for transfers.
\end{lemma}
\begin{proof}
  We have
  \begin{align*}
    N^{mnk}_{nk} V^{nk}_n(x)_{\ell\mid mnk}
      &= \psi^{\ell/g}(V^{nk}_n(x)_g)^{mg/\ell} \qquad g=\gcd(nk,\ell)\\
      &= [g\mid n]\psi^{\ell/g}(kx_g)^{mg/\ell}
  \end{align*}
  From \cite[Theorem 2.8]{MazurTambara}, the identity we need to establish is
  \[ N^{mnk}_{nk} V^{nk}_n(x) = \sum_{\substack{d\mid mnk\\n=\gcd(nk,d)}} V^{mnk}_d \sum_\omega N^d_n(\omega) \]
  We expand
  \begin{align*}
    &\phantom{{}=} \sum_{\substack{d\mid mnk\\n=\gcd(nk,d)}} V^{mnk}_d \sum_\omega N^d_n(\omega)_{\ell\mid mnk}\\
    &= \sum_{\substack{j\mid m\\\gcd(j,k)=1}}[\ell\mid nj]\frac{mk}j\sum_\omega N^{nj}_n(\omega)_\ell\\
    &= \sum_{\substack{j\mid m\\\gcd(j,k)=1}}[\ell\mid nj]\frac{mk}j\sum_\omega \psi^{\ell/h}(\omega_h)^{jh/\ell}\qquad h=\gcd(n,\ell)\\
    \intertext{Note that $[g\mid n]$ is nonzero if and only if $g=h$. Similarly, $\ell\mid nj$ and $\gcd(j,k)=1$ together imply $g=h$. Therefore, we may assume $g=h$, and continue expanding}
    &= \psi^{\ell/g}\sum_{\substack{j\mid m\\\gcd(j,k)=1}}\frac{mk}j \sum_\omega (\omega_g^{g/\ell})^j\\
    &= \psi^{\ell/g}(kx_g)^{mg/\ell}\qedhere
  \end{align*}
\end{proof}

We have proven:

\begin{theorem}
  Let $(s,\psi)$ be a transversal $\T$-Green $\lambda$-GNS on $D$, and let $\mathcal S=\{n\in\N \mid s(n)\in D^\times\}$. Then $N^{mn}_n$ gives a well-defined, multiplicative function
  \[ D/s(n) \to D/s(mn)[\Phi_k(s)^{-1} \mid \gcd(n,k)\in\mathcal S] \]
  which satisfies all the identities of a Tambara functor.

  In particular, if $s(1)\notin D^\times$, then $\m D^-$ is a Tambara functor, while $\m D$ is an \emph{incomplete} Tambara functor \cite{IncompleteTambara} for the indexing system $\bfP$.
\end{theorem}
\begin{remark}
  Norm maps such as $N^6_2\colon D/(2)_q \to D/(2)_{q^3}$ are not captured by the existing framework of incomplete Tambara functors. We intend to give a more permissive framework which captures this example in future work.
\end{remark}

\subsection{Lifts of the Norm}
\label{sub:norm-lifts}

We now study \emph{lifts} of our norm maps $N^{mn}_n$. We begin with the ``$s$-twisted powers'' $(x-y)^n_s$.

\begin{definition}[{\cite[Definition 2.3.3]{GNS}}]
  Let $s$ be a GNS on the ring $D$. The \emph{$s$-derivative} is the $D$-linear map $\nabla_s\colon D[x]\to D[x]$ given on monomials by $\nabla_s(x^n) = s(n)x^{n-1}$.
\end{definition}

\begin{theorem}[{\cite[Definition 2.3.5, Proposition 2.3.8]{GNS}}]
  Let $s$ be a GNS on the ring $D$. There is a unique polynomial $(x-y)^n_s\in D[x,y]$ such that:
  \begin{itemize}
    \item $(x-y)^0_s = 1$;
    \item $(x-x)^n_s = 0$;
    \item $\nabla_{s,x}(x-y)^n_s = s(n)(x-y)^{n-1}_s$.
  \end{itemize}
\end{theorem}

\begin{warning}
  Beware that despite the notation, $(x-y)^n_s$ is a function of both $x$ and $y$, not merely of $x-y$. However, this will be partially alleviated by Theorem \ref{thm:s-lift}.
\end{warning}

\begin{theorem}[$s$-binomial theorem, {\cite[Theorem 2.3.7]{GNS}}]
  Let $s$ be a GNS on $D$. For all $x,y\in D$, we have
\[ (x-y)^n_s = \sum_{k=0}^k {\binom nk}_s (x-0)^{n-k}_s (0-y)^k_s. \]
\end{theorem}

\begin{lemma}
  \label{lem:sbinom-1}
  For all $d\mid n$, we have
  \[ (0-1)^n_s \equiv (-1)^{n/d} \bmod{\Phi_d(s)}. \]
\end{lemma}
\begin{proof}
  Using the $s$-binomial theorem and the $s$-Lucas theorem, we have
  \begin{align*}
    0 &= (1 - 1)^n_s\\
      &= \sum_{k=0}^n {\binom nk}_s (0-1)^k_s\\
      &\equiv \sum_{\substack{k=0\\d\mid k}}^n {\binom{n/d}{k/d}} (0-1)^k_s \quad\bmod {\Phi_d(s)}\\
      &= \sum_{k'=0}^{n'} \binom{n'}{k'} (0-1)^{dk'}_s
  \end{align*}
  where we set $n=dn'$, $k=dk'$. In other words, the sequence $a_\ell=(0-1)^{d\ell}_s\pmod{\Phi_d(s)}$ satisfies the recurrence relation
  \[ 0 = \sum_{\ell=0}^m \binom m\ell a_\ell \]
  for all $m\ge 0$. The claim follows since $a_\ell=(-1)^\ell$ also satisfies this recurrence relation.
\end{proof}
\begin{corollary}
  \label{cor:sbinom-1}
  For all $d\mid n$, we have
  \[ (0-y)^n_s \equiv (-y^d)^{n/d} \bmod{\Phi_d(s)}. \]
\end{corollary}

\begin{theorem}
  \label{thm:s-lift}
  Let $(s,\psi)$ be a transversal $\lambda$-GNS on a ring $D$, and let $x,y\in D$ be elements of rank one. Then for all $n\in\N$, the $s$-twisted power $(x-y)^n_s$ lifts the norm $N^n_1(x-y)$: we have
  \[ (x-y)^n_s \bmod{s(n)} = N^n_1(x-y\bmod{s(1)}). \]
\end{theorem}
\begin{proof}
  Since $x$ and $y$ are of rank one, it is equivalent to show that
  \[ (x-y)^n_s \equiv (x^d-y^d)^{n/d} \bmod{\Phi_d(s)} \]
  for all $d\mid n$. Using Corollary \ref{cor:sbinom-1} (and imitating its proof), we get
  \begin{align*}
    (x-y)^n_s
      &= \sum_{k=0}^n {\binom nk}_s x^{n-k} (0-y)^k_s\\
      &\equiv \sum_{k'=0}^{n/d} \binom{n/d}{n/k} x^{d(n/d-k')} (-y^d)^{k'}\\
      &= (x^d-y^d)^{n/d}.\qedhere
  \end{align*}
\end{proof}

\begin{remark}
  At odd primes, something similar is true for the expression $(x+y)^n_s$. But beware that
  \[ (x+y)^2_q = x+(2)_q xy + qy^2 \equiv x^2-y^2 \bmod{(2)_q}. \]
\end{remark}

\begin{remark}
  Hill has defined a cotangent complex for Tambara functors \cite{HillTambaraAQ}. Roughly speaking, his $G$-derivative \cite[Definition 4.1]{HillTambaraAQ} satisfies
  \[ \nabla_G(x^{H/K}) = H/K\cdot \text{``}x^{H/K}\dlog x\text{''}. \]
  In view of Theorem \ref{thm:s-lift}, this looks very similar to
  \[ \nabla_{s,x}(x-y)^n_s = s(n)(x-y)^{n-1}_s. \]
  But the $G$-derivative is a levelwise derivation with additional compabilities, while the $s$-derivative is not a derivation. It would be extremely interesting to understand the relation between the ``$G$-de Rham complex'' (yet to be defined beyond $G\Omega^1$) and the $s$-de Rham complex.
\end{remark}

We now turn to finding an integral analogue of Lemma \ref{lem:n-lift-p}, i.e.\ finding lifts $\~N^{mn}_n(f)$ of $N^{mn}_n(f)$ for arbitrary $f$. Although our reinterpretation of the $\theta_n$ operations points the way to defining operations $\vartheta_n$ with $\theta_n = \vartheta_{p^n}$, we have so far been unable to construct $\vartheta_{12}$.

\begin{conjecture}
  \label{conj:vartheta}
  Every $\lambda$-ring admits operations $\vartheta_i$ such that
  \[ \psi^m(f^n) = f^{mn} + \sum_{\substack{d\mid mn\\d\nmid n}} d\vartheta_d(f)^{\psi^{mn/d}}. \]
\end{conjecture}

\begin{proposition}
  \label{prop:vartheta-pq}
  Conjecture \ref{conj:vartheta} is true when $mn$ is a power of a prime or the product of two primes.
\end{proposition}
\begin{proof}
  The $(mn,1)$ case follows from the $(m,n)$ and $(n,1)$ cases by applying $\psi^m$ to $\psi^n(f)$. The case $m=p^r$ is already known from the $p$-typical case, so that $\theta_n = \vartheta_{p^n}$.

  Suppose $m=pq$ for two distinct primes $p$ and $q$. Then
  \begin{align*}
    \psi^p(f)
      &= f^p + p\vartheta_p(f)\\
    \psi^{pq}(f)
      &= \psi^q(f^p) + p\vartheta(f)^{\psi^q}\\
    &=
      (f^q + q\vartheta_q(f))^p + p\vartheta(f)^{\psi^q}\\
    &\equiv
      f^{pq} + q^p\vartheta_q(f)^p + p\vartheta(f)^{\psi^q}\bmod pq\\
    \intertext{Since $p$ is prime, and $p\ne q$, we have $q^{p-1}\equiv 1\bmod p$, and thus $q^p\equiv q\bmod pq$. We also have $\vartheta_q(f)^p\equiv\vartheta_q(f)^{\psi^p} \bmod p$. Consequently, we get}
    \psi^{pq}(f)
      &\equiv f^{pq} + p\vartheta_p(f)^{\psi^q} + q\vartheta_q(f)^{\psi^p} \bmod pq.
    \qedhere
  \end{align*}
\end{proof}

\begin{theorem}
  \label{thm:vartheta-lift}
  Assuming Conjecture \ref{conj:vartheta} (for $n=1$), the operation
  \[ \~N^{mn}_n(f) = f^{\psi^m} - \sum_{\substack{d\mid m\\d\ne1}} \~V^{mn}_{mn/d} \vartheta_d(f)^{\psi^{m/d}}, \]
  where $\~V^{mn}_{mn/d}(f) = \frac{s(mn)}{s(mn/d)}f$, is a lift of $N^{mn}_n(f\bmod s(n))$.
\end{theorem}
\begin{proof}
  We must check that $\~N^{mn}_n$ satisfies the congruence
  \[ \~N^{mn}_n(f) \equiv \psi^{k/\ell}(f)^{m\ell/k}\bmod \Phi_k(s), \qquad \ell=\gcd(n,k) \]
  for all $k\mid mn$. We compute
  \begin{align*}
    \~N^{mn}_n(f)
      &= f^{\psi^m} - \sum_{\substack{d\mid m\\d\ne1}} \~V^{mn}_{mn/d} \vartheta_d(f)^{\psi^{m/d}}\\
      &\equiv f^{\psi^m} - \sum_{\substack{d\mid m\\d\ne1}} [k\mid\tfrac{mn}d]d \vartheta_d(f)^{\psi^{m/d}} \bmod \Phi_k(s)\\
      \intertext{The condition $k\mid\frac{mn}d$ is equivalent to $d\mid\frac{mn}k$, which combined with $d\mid m$ gives $d\mid\frac m{k/\ell}$. Continuing to rewrite, we get}
      &= f^{\psi^m} - \sum_{d\mid\frac{mk}\ell} d \vartheta_d(f)^{\psi^{m/d}}\\
      &= \psi^{k/\ell}(f^{\psi^{m\ell/k}} - \sum_{d\mid\frac{mk}\ell} d \vartheta_d(f)^{\psi^{m/d}})\\
      &= \psi^{k/\ell}(f^{m\ell/k})\\
      &= \psi^{k/\ell}(f)^{m\ell/k}.\qedhere
  \end{align*}
\end{proof}

%% file: tambara.bbl
\newcommand{\etalchar}[1]{$^{#1}$}
\providecommand{\bysame}{\leavevmode\hbox to3em{\hrulefill}\thinspace}
\providecommand{\MR}{\relax\ifhmode\unskip\space\fi MR }
% \MRhref is called by the amsart/book/proc definition of \MR.
\providecommand{\MRhref}[2]{%
  \href{http://www.ams.org/mathscinet-getitem?mr=#1}{#2}
}
\providecommand{\href}[2]{#2}

%% file: tambara.bbl
\begin{thebibliography}{AMGR22}

\bibitem[AB19]{AClBTrace}
Johannes Ansch\"utz and Artur C\'esar-Le Bras, \emph{The $p$-completed
  cyclotomic trace in degree 2}, \url{https://arxiv.org/abs/1907.10530v1}.

\bibitem[ABG{\etalchar{+}}18]{ABGHLM}
Vigleik Angeltveit, Andrew~J. Blumberg, Teena Gerhardt, Michael~A. Hill, Tyler
  Lawson, and Michael~A. Mandell, \emph{Topological cyclic homology via the
  norm}, Documenta mathematica \textbf{23} (2018).

\bibitem[AMGR21]{AMGRMackey}
David Ayala, Aaron Mazel-Gee, and Nick Rozenblyum, \emph{{Derived Mackey
  functors and $C_{p^n}$-equivariant cohomology}},
  \url{https://arxiv.org/abs/2105.02456}.

\bibitem[AMGR22]{AMGRStratified}
\bysame, \emph{Stratified noncommutative geometry},
  \url{https://arxiv.org/abs/1910.14602}.

\bibitem[AN21]{TCart}
Benjamin Antieau and Thomas Nikolaus, \emph{Cartier modules and cyclotomic
  spectra}, Journal of the American Mathematical Society \textbf{34} (2021),
  no.~1, 1--78.

\bibitem[Ang15]{AngeltveitNorm}
Vigleik Angeltveit, \emph{{The norm map of Witt vectors}}, Comptes Rendus
  Mathematique \textbf{353} (2015), no.~5, 381--386.

\bibitem[Ant23]{AntieauSWitt}
Benjamin Antieau, \emph{{Spherical Witt vectors and integral models for
  spaces}}, \url{https://arxiv.org/abs/2308.07288v1}.

\bibitem[BH18]{IncompleteTambara}
Andrew Blumberg and Michael Hill, \emph{{Incomplete Tambara functors}},
  Algebraic {\&} Geometric Topology \textbf{18} (2018), no.~2, 723--766.

\bibitem[BH21]{BiIncomplete}
\bysame, \emph{{Bi-incomplete Tambara functors}},
  \url{https://arxiv.org/abs/2104.10521}.

\bibitem[BL22]{APC}
Bhargav Bhatt and Jacob Lurie, \emph{Absolute prismatic cohomology},
  \url{https://arxiv.org/abs/2201.06120}.

\bibitem[Bru05]{BrunWitt}
Morten Brun, \emph{{Witt vectors and Tambara functors}}, Advances in
  Mathematics \textbf{193} (2005), no.~2, 233--256.

\bibitem[BS22]{Prismatic}
Bhargav Bhatt and Peter Scholze, \emph{{Prisms and prismatic cohomology}},
  Annals of Mathematics \textbf{196} (2022), no.~3, 1135 -- 1275.

\bibitem[BSY22]{ChromaticNullstellensatz}
Robert Burklund, Tomer~M. Schlank, and Allen Yuan, \emph{{The Chromatic
  Nullstellensatz}}, \url{https://arxiv.org/abs/2207.09929}.

\bibitem[{\v C}S]{CSPurity}
K{\c e}stutis {\v C}esnavi{\v c}ius and Peter Scholze, \emph{Purity for flat
  cohomology}, Annals of Mathematics, To appear.

\bibitem[DM23]{GNS}
S.~K. Devalapurkar and M.~L. Misterka, \emph{{Generalized $n$-series and de
  Rham complexes}},
  \url{https://sanathdevalapurkar.github.io/files/fgls-and-dR-complexes.pdf}.

\bibitem[DPM22]{DMPPolynomial}
Emanuele Dotto, Irakli Patchkoria, and Kristian~Jonsson Moi, \emph{{Witt
  Vectors, Polynomial Maps, and Real Topological Hochschild Homology}}, Annales
  scientifiques de l{\textquotesingle}{\'{E}}cole Normale Sup{\'{e}}rieure
  \textbf{55} (2022), no.~2, 473--535.

\bibitem[DS88]{DSBurnside}
Andreas~W.M. Dress and Christian Siebeneicher, \emph{{The Burnside ring of
  profinite groups and the Witt vector construction}}, Advances in Mathematics
  \textbf{70} (1988), no.~1, 87--132.

\bibitem[Gla17]{Glasman}
Saul Glasman, \emph{{Stratified categories, geometric fixed points and a
  generalized Arone-Ching theorem}}, \url{https://arxiv.org/abs/1507.01976}.

\bibitem[GSQ23]{GLQqCrys}
Michel Gros, Bernard~Le Stum, and Adolfo Quir{\'o}s, \emph{Twisted differential
  operators and q-crystals}, p-adic Hodge Theory, Singular Varieties, and
  Non-Abelian Aspects (Cham) (Bhargav Bhatt and Martin Olsson, eds.), Springer
  International Publishing, 2023, pp.~183--238.

\bibitem[Hil17]{HillTambaraAQ}
Michael~A. Hill, \emph{{On the Andr\'e-Quillen homology of Tambara functors}},
  Journal of Algebra \textbf{489} (2017), 115--137.

\bibitem[Hil20]{HillHandbook}
Michael~A.\ Hill, \emph{Equivariant stable homotopy theory}, Handbook of
  Homotopy Theory (Haynes Miller, ed.), Chapman and Hall/{CRC}, January 2020,
  pp.~699--756.

\bibitem[HM97]{HMFinite}
Lars Hesselholt and Ib~Madsen, \emph{{On the $K$-theory of finite algebras over
  Witt vectors of perfect fields}}, Topology \textbf{36} (1997), no.~1, 29 --
  101.

\bibitem[HM19]{MazurTambara}
Michael~A. Hill and Kristen Mazur, \emph{{An equivariant tensor product on
  Mackey functors}}, Journal of Pure and Applied Algebra \textbf{223} (2019),
  no.~12, 5310--5345.

\bibitem[Kal11]{KaledinMackey}
Dmitry~Borisovich Kaledin, \emph{{Derived Mackey functors}}, {Moscow
  Mathematical Journal} \textbf{11} (2011), no.~4, 723--803.

\bibitem[Lur18]{Ell2}
Jacob Lurie, \emph{{Elliptic cohomology II: Orientations}},
  \url{https://www.math.ias.edu/~lurie/papers/Elliptic-II.pdf}.

\bibitem[Lur21]{LurieMOPrisms}
\bysame, \emph{What are the potential applications of perfectoid spaces to
  homotopy theory?}, MathOverflow, 2021,
  \url{https://mathoverflow.net/questions/273352/what-are-the-potential-applications-of-perfectoid-spaces-to-homotopy-theory\#comment985403_386521}.

\bibitem[Mol20]{Molokov}
Semen Molokov, \emph{{Prismatic cohomology and de Rham-Witt forms}},
  \url{https://arxiv.org/abs/2008.04956}.

\bibitem[NS18]{NikolausScholze}
Thomas {Nikolaus} and Peter {Scholze}, \emph{{On topological cyclic homology}},
  {Acta Math.} \textbf{221} (2018), no.~2, 203--409.

\bibitem[PSW22]{SpecDerivedMackey}
Irakli Patchkoria, Beren Sanders, and Christian Wimmer, \emph{{The spectrum of
  derived Mackey functors}}, Transactions of the American Mathematical Society
  \textbf{375} (2022), no.~06, 4057--4105.

\bibitem[Sch18]{GlobalHomotopy}
Stefan Schwede, \emph{Global homotopy theory}, 2018,
  \url{https://arxiv.org/abs/1802.09382}.

\bibitem[Sul20]{SulSliceTHH}
Yuri J.~F. Sulyma, \emph{{A slice refinement of B\"okstedt periodicity}},
  \url{https://arxiv.org/abs/2007.13817}.

\bibitem[Sul23]{SulSlopes}
\bysame, \emph{{Floor, ceiling, slopes, and $K$-theory}}, {Annals of
  $K$-Theory} \textbf{8} (2023), no.~3, 331--354.

\bibitem[Yan23]{NormedEooRingsCp}
Lucy Yang, \emph{{On normed $\mathbb E_\infty$-rings in genuine equivarant
  $C_p$-spectra}}, \url{https://arxiv.org/abs/2308.16107}.

\end{thebibliography}
